\newif\ifpictures
\picturestrue

\documentclass[11pt]{amsart}
\usepackage{amssymb,amsmath}
\usepackage{amsopn}
\usepackage{xspace}
\usepackage[dvips]{graphicx}
 
\usepackage[rgb]{xcolor}
\definecolor{DarkBlue}{rgb}{0,0.2,0.6}
\usepackage[colorlinks=true, citecolor=DarkBlue, linkcolor=DarkBlue, urlcolor=DarkBlue]{hyperref}
\usepackage[arrow,matrix,curve]{xy}
\usepackage{setspace}
\usepackage{tikz,color}
\usepackage{wasysym}

\newcommand{\struc}[1]{{\color{DarkBlue} #1}}

\numberwithin{equation}{section}
\newtheorem{thm}{Theorem}
\newtheorem{prop}[thm]{Proposition}
\newtheorem{lemma}[thm]{Lemma}
\newtheorem{cor}[thm]{Corollary}

\newtheorem{prob}[thm]{Problem}

\numberwithin{thm}{section}
\theoremstyle{definition}
\newtheorem{defn}[thm]{Definition}
\newtheorem*{defn*}{Definition}
\newtheorem{exa}[thm]{Example}
\newtheorem*{exa*}{Example}

\newcounter{FNC}[page]
\def\newfootnote#1{{\addtocounter{FNC}{2}$^\fnsymbol{FNC}$%
     \let\thefootnote\relax\footnotetext{$^\fnsymbol{FNC}$#1}}}

\newcommand{\C}{\mathbb{C}}

\newcommand{\F}{\mathbb{F}}
\newcommand{\K}{\mathbb{K}}
\newcommand{\N}{\mathbb{N}}
\renewcommand{\P}{\mathbb{P}}
\newcommand{\Q}{\mathbb{Q}}
\newcommand{\R}{\mathbb{R}}

\newcommand{\Z}{\mathbb{Z}}

\newcommand\cA{{\ensuremath{\mathcal{A}}}\xspace}
\newcommand\cB{{\ensuremath{\mathcal{B}}}\xspace}
\newcommand\cC{{\ensuremath{\mathcal{C}}}\xspace}

\newcommand\cM{{\ensuremath{\mathcal{M}}}\xspace}

\newcommand\cR{{\ensuremath{\mathcal{R}}}\xspace}
\newcommand\cS{{\ensuremath{\mathcal{S}}}\xspace}
\newcommand\cT{{\ensuremath{\mathcal{T}}}\xspace}

\newcommand\cV{{\ensuremath{\mathcal{V}}}\xspace}

\newcommand{\eps}{\varepsilon}

\newcommand{\alp}{\alpha}
\newcommand{\lam}{\lambda}

\newcommand{\lf}{\left}
\newcommand{\ri}{\right}
\newcommand{\ra}{\rightarrow}

\newcommand{\Lera}{\Leftrightarrow}
\newcommand{\ovl}{\overline}

\DeclareMathOperator{\trop}{trop}

\DeclareMathOperator{\Arg}{Arg}
\DeclareMathOperator{\Log}{Log}
\DeclareMathOperator{\New}{New}

\DeclareMathOperator{\vol}{vol}
\DeclareMathOperator{\re}{re}
\DeclareMathOperator{\im}{im}

\DeclareMathOperator{\ord}{ord}

\DeclareMathOperator{\RE}{Re}
\DeclareMathOperator{\IM}{Im}

\DeclareMathOperator{\val}{val}

\DeclareMathOperator{\Ker}{Ker}
\DeclareMathOperator{\lt}{lt}
\DeclareMathOperator{\area}{area}
\DeclareMathOperator{\archtrop}{archtrop}

\DeclareMathOperator{\SL}{SL}

\newcommand{\Fa}[1][|\mathbf{z}|]{f^{#1}}

\def\endenvi{\hfill$\hexagon$}

\title{The Boundary of Amoebas}
\author{Franziska Schroeter}

\author{Timo de Wolff}

\address{Franziska Schroeter, Universit\"at Hamburg, Analysis und Differentialgeometrie, Fachbereich Mathematik, Bundesstr. 55,
 20146 Hamburg, Germany \medskip \newline 
\hspace*{8pt} Timo de Wolff, Texas A\&M University, Department of 
Mathematics, College Station, TX 77843-3368, 
 USA \medskip
}

\email{franziska.schroeter@uni-hamburg.de, dewolff@tamu.math.edu}


\begin{document}

\begin{abstract}
The computation of amoebas has been a challenging open problem for the last dozen years. The most natural 
approach, namely to compute an amoeba via its boundary has not been practical so far since only a superset of the boundary, the contour, is understood in theory and computable in practice.

We define and characterize the extended boundary of an amoeba, which is sensitive to some degenerations that the topological boundary does not detect. Our description of the extended boundary also allows us
to distinguish between the contour and the boundary. This gives rise not only to new 
structural results in amoeba theory, but in particular allows us to compute hypersurface 
amoebas via their boundary in any dimension. In dimension two this can be done
using Gr\"obner bases alone.

We introduce the concept of amoeba bases, which are sufficient for understanding the amoeba of an ideal. We show that our characterization of the boundary is essential for the computation of these amoeba bases and we illustrate the potential of this concept by constructing amoeba bases for linear systems of equations.
\end{abstract}

\maketitle

\section{Introduction}
\label{Sec:Introduction}
Let $\struc{f} \in \struc{\C[\mathbf{z}^{\pm 1}]} = \C[z_1^{\pm 1},\ldots,z_n^{\pm 1}]$ be a multivariate Laurent
polynomial defining a non-singular hypersurface $\struc{\cV(f)} \subset \struc{(\C^*)^n} = (\C \setminus \{0\})^n$. The
\struc{\textit{amoeba} $\cA(f)$} of $f$, introduced by Gelfand, Kapranov, and Zelevinsky in 1993
\cite{Gelfand:Kapranov:Zelevinsky}, is the image of $\cV(f)$ under the $\Log$ absolute value map given
by \begin{eqnarray}
\label{Equ:LogMap}
	\struc{\Log|\cdot|}: \ \lf(\C^*\ri)^n \to \R^n, \quad (z_1,\ldots,z_n) \mapsto (\log|z_1|,
\ldots, \log|z_n|) \, .
\end{eqnarray}

\noindent Amoebas are objects with an amazing amount of structural properties, which have
been intensively studied during the last 20 years. In particular, amoebas reveal an
intrinsic connection between classical algebraic geometry and tropical geometry
\cite{Einsiedler:Kapranov:Lind,Maclagan:Sturmfels,Maslov}. Moreover, they
appear in various other fields of mathematics such as
complex analysis \cite{Forsberg:Passare:Tsikh,Passare:Rullgard:Spine}, real algebraic 
geometry \cite{Iliman:deWolff:Circuits,Mikhalkin:Annals}, and statistical thermodynamics \cite{Passare:Pochekutov:Tsikh}. Expository writings on amoeba theory see \cite{deWolff:Diss,Gelfand:Kapranov:Zelevinsky,Mikhalkin:Survey,Passare:Tsikh:Survey,Rullgard:Diss}.

Gelfand, Kapranov, and Zelevinsky observed that every connected component of the complement
of an amoeba is an open, convex set \cite[p.\:195, Prop.\:1.5 and 
Cor.\:1.6]{Gelfand:Kapranov:Zelevinsky}; see also \cite[Prop. 
1.2]{Forsberg:Passare:Tsikh}.
This means that the amoeba itself has a boundary (note that the amoeba can also be 
compactified if necessary, e.g. via toric compactification 
\cite{Gelfand:Kapranov:Zelevinsky,Mikhalkin:Annals}). With exception of the 
linear case \cite{Forsberg:Passare:Tsikh} this boundary is very hard to describe and no 
explicit characterization is known so far.

For many applications it is essential to compute or at least 
approximate amoebas. The computation and approximation of amoebas was pioneered
by Theobald \cite{Theobald:ComputingAmoebas}, where an approximation of the
\struc{\textit{contour}}, a superset of the topological boundary, in the plane is described, using
a result by Mikhalkin; see Theorem \ref{Thm:MikhalkinContour}. Later, approximation methods of amoebas based on (algebraic)
certificates were given by Purbhoo \cite{Purbhoo} using iterated
resultants, by Theobald and the second author \cite{Theobald:deWolff:SOS} via semidefinite
programming (SDP) and sums of squares (SOS), and by Avenda\~{n}o, Kogan,
Nisse and Rojas \cite{Nisse:Rojas:et:al} via tropical geometry. We compare these
approaches in Section \ref{Sec:Computation} in detail. Here, we only point out that
none of these methods can be used to compute the boundary of an amoeba, decide 
membership of a point in an amoeba exactly (only non-membership can be certified so far), or give 
general degree bounds on certificates. In this article we provide an explicit 
characterization of the \struc{\textit{extended boundary}} of an amoeba (Corollary 
\ref{Cor:FiberComputation}), which can be checked algorithmically. Particularly, membership in the extended boundary, the contour and the amoeba itself can be decided exactly. Moreover, the entire boundary of an amoeba, and therefore also the amoeba itself, can be approximated in this way, see Theorem \ref{Thm:AmoebaComputation}.\\

The \struc{\textit{contour} $\cC(f)$} of the amoeba $\cA(f)$ is defined as the $\Log|\cdot|$ image of the set of 
points of $\cV(f)$, which are critical with respect to the $\Log|\cdot|$ map, see Section \ref{SubSec:Contour} for details. Mikhalkin's results \cite{Mikhalkin:Annals, Mikhalkin:PairsOfPants} state
that the critical points of $\Log|\cdot|$ restricted to $\cV(f)$, and thus also the contour $\cC(f)$, are given by the
points with a real image under the \textit{logarithmic Gau\ss{} map}. The 
\struc{\textit{logarithmic Gau\ss{} map}}, introduced by Kapranov \cite{Kapranov:LogGaussMap}, is a 
composition of a branch of the holomorphic logarithm of
each coordinate with the usual Gau\ss{} map, see more details in Section \ref{SubSec:GaussMap}. It is 
given by
\begin{eqnarray*}
	\struc{\gamma}: \cV(f) \ra \P_{\C}^{n-1}, \quad \mathbf{z}=(z_1,\ldots,z_n) \mapsto
\lf(z_1 \cdot \frac{\partial f}{\partial z_1}(\mathbf{z}):\cdots:z_n \cdot \frac{\partial
f}{\partial
z_n}(\mathbf{z})\ri).
\end{eqnarray*}

For a given hypersurface $\cV(f)$ we define the set $S(f)$ by 
\begin{eqnarray*}
	\struc{S(f)}	& =	& \{\mathbf{z} \in \cV(f) \ : \ \gamma(\mathbf{z}) \in
\P_{\R}^{n-1} \subset \P_{\C}^{n-1}\}.
\end{eqnarray*}

\begin{thm}[Mikhalkin \cite{Mikhalkin:Annals,Mikhalkin:PairsOfPants}]
Let $f \in \C[\mathbf{z}^{\pm 1}]$ with $\cV(f) \subset (\C^*)^n$. Then the set of
critical points of the $\Log|\cdot|$ map equals $S(f)$ and thus it follows that $\cC(f)=\Log|S(f)|$.
\label{Thm:MikhalkinContour}
\end{thm}

\noindent For real $f$, this implies that the \struc{\textit{real locus} $\cV_\R(f)$}, i.e., the
set of real points in $\cV(f)$, is always contained in the contour $\cC(f)$.

Concerning the boundary of an amoeba, this theorem yields the following corollary. A point $\mathbf{w} \in \R^n$ may only be a boundary point of an amoeba $\cA(f)$ if there exists a point in the intersection of its fiber $\F_{\mathbf{w}}=\{\mathbf{z}\in
(\C^*)^n:~\Log|\mathbf{z}|=\mathbf{w}\}$ and the variety $\cV(f)$, which belongs to the
set $S(f)$.

\begin{cor}[Mikhalkin]
Let $f \in \C[\mathbf{z}^{\pm 1}]$ with $\cV(f) \subset (\C^*)^n$ and let $\mathbf{w} \in
\R^n$. Then
\begin{eqnarray*}
	& & \mathbf{w} \in \partial \cA(f) \, \text{ implies that } \, \F_{\mathbf{w}}
\cap \cV(f) \cap S(f) \neq \emptyset.
\end{eqnarray*}
\label{Cor:MikhalkinBoundary}
\end{cor}

\vspace*{-12pt}

We define the \struc{\textit{extended boundary}} of an amoeba as follows. Let $f \in \C[\mathbf{z}^{\pm 1}]$ with corresponding amoeba $\cA(f)$. We call a point $\mathbf{w} \in \cA(f)$ an \struc{\textit{extended boundary point}} if it can be transferred to the complement of the amoeba by a perturbation of the coefficients of the defining polynomial $f$. We denote the extended boundary of $\cA(f)$ by $\struc{\partial^{e}\cA(f)}$. A formal definition of the extended boundary is given in Definition \ref{Def:ExtendedBoundary}.

 Note that $\partial^{e} \cA(f)$ contains the topological boundary $\partial \cA(f)$ of an amoeba, but $\partial \cA(f)$ can be a strict subset of $\partial^{e} \cA(f)$. In Section \ref{SubSec:Amoebas} we also motivate the concept of the extended boundary of an amoeba; see particularly Example \ref{Ex:Boundary}. 

We give an explicit characterization of the extended boundary of
amoebas up to singular points of the contour by strengthening Corollary
\ref{Cor:MikhalkinBoundary}. More precisely, we show that a point $\mathbf{w} \in \R^n$ may be a boundary point of an amoeba
$\cA(f)$ only if \textit{every} point in the (non-empty) intersection of its
fiber $\F_{\mathbf{w}}$ and the variety $\cV(f)$ belongs to the set $S(f)$. Furthermore,
we show that this condition is also sufficient if $\mathbf{w}$ is a non-singular point of
the contour.

\begin{thm}
Let $f \in \C[\mathbf{z}^{\pm 1}]$ with $\cV(f) \subset (\C^*)^n$ smooth and let
$\mathbf{w} \in \R^n$.
\begin{eqnarray}
	& & \text{ If } \mathbf{w} \in \partial^{e} \cA(f) \, \text{, then } \,
\F_{\mathbf{w}} \cap \cV(f) \subseteq S(f).
\label{Equ:MainTheorem}
\end{eqnarray}
If in addition $\mathbf{w}$ is a smooth point of the contour $\cC(f)$, then the
implication in \eqref{Equ:MainTheorem} is an equivalence.
\label{Thm:Main1}
\end{thm}

In the case of the topological boundary $\partial \cA(f)$, the implication \eqref{Equ:MainTheorem} follows easily from the implicit function theorem. The extension to $\partial^{e} \cA(f)$, however, requires a more sophisticated proof.

Theorem \ref{Thm:Main1} characterizes those points of the contour of an amoeba that belong to the extended boundary of the amoeba. However, to compute an amoeba $\cA(f)$ via its extended boundary, we need to transform this characterization into an algorithmic test for points of the contour $\cC(f)$. 
 
Thus, we use Theorem \ref{Thm:Main1} to derive a second, less abstract description of the extended boundary which is much more useful from a computational point of view.
 
The precise statement is given in Theorem \ref{Thm:BoundaryArbitraryDimension}. We summarize the result here as follows.

\begin{cor}
Let $f \in \C[\mathbf{z}^{\pm 1}]$ with $\cV(f) \subset (\C^*)^n$ smooth and
$\mathbf{w}$ a smooth point of the contour $\cC(f)$. Then $\mathbf{w}$ is a 
point of the extended boundary of the amoeba $\cA(f)$ if and only if every point $\mathbf{v} \in \cV(f) \cap \F_{\mathbf{w}}$ has
\begin{itemize}
 \item multiplicity greater than one for $n = 2$, and
 \item $\cV(f) \cap \F_{\mathbf{w}}$ is finite for $n \geq 3$.
\end{itemize}
This is the case if and only if every real zero of a certain
ideal $I \subset \R[x_1,\ldots,x_n,y_1,\ldots,y_n]$ determined by $f$ and $\mathbf{w}$
\begin{itemize}
 \item has multiplicity greater than one for $n = 2$,
 \item and the real locus $\cV_\R(I)$ is finite for $n \geq 3$.
\end{itemize}
\label{Cor:Main1}
\end{cor}

This result goes beyond a description of the contour and the (extended) boundary of amoebas. It provides new geometric insight since it allows to interpret the contour of an amoeba as the boundary of the full dimensional cells a particular decomposition of the ambient space of an amoeba, see Theorem \ref{Thm:BettiDecompositionContour}. This \struc{\textit{Betti decomposition}} is given by the $0$-th Betti number of the intersection of the original hypersurface with the fiber over each point in the amoeba space.

For $n = 3$ we give bounds on the number of connected components of $\cV(f) \cap
\F_{\mathbf{w}}$ depending on the degree of $f$. For $n = 2$ the entire set $\cV(f) \cap 
\F_{\mathbf{w}}$ is generically finite and we can  compute its cardinality. For further details see 
Section \ref{Sec:GenericallySufficient}, particularly Theorem 
\ref{Thm:BettiUpperBound}.\\

Note that Theorem \ref{Thm:BoundaryArbitraryDimension} allows us to recover the description of amoebas of linear polynomials by Forsberg,
Passare and Tsikh \cite{Forsberg:Passare:Tsikh}, see also Proposition \ref{Prop:LinearCase}. Also, Theorem \ref{Thm:BoundaryArbitraryDimension} implies that for every Harnack curve \cite{Harnack} the contour and the boundary of the corresponding
amoeba coincide, which had been proven earlier by Mikhalkin and Rullg{\aa}rd \cite{Mikhalkin:Rullgard}, see also Proposition \ref{Prop:Harnack}.\\

\medskip

Based on these theoretical results we provide a new algorithmic approach for the initial 
problem: the computation of amoebas. This approach has the following key properties, 
see Section \ref{Sec:Computation} for details.

\begin{thm}
Suppose $\textnormal{i}$ satisfies $\textnormal{i}^2 = -1$. Let $f \in \Q[\textnormal{i}][\mathbf{z}^{\pm 1}]$  and
$\mathbf{v} \in \Q^n$. Then we can compute the contour and the boundary of
$\cA(f)$ as well as decide membership of $\Log|\mathbf{v}| \in \cA(f)$ algorithmically.
\label{Thm:AmoebaComputation}
\end{thm}

This approach can be implemented efficiently in dimension two and the
computation uses only Gr\"obner basis methods for zero-dimensional ideals. The Betti decomposition
mentioned above can be approximated with the same methods. An implementation is also possible in higher dimensions, but then the computation requires quantifier elimination methods, which are known to be computationally hard; see e.g. \cite{Basu:Pollack:Roy}. We have 
implemented a prototype version of both the boundary computation in dimension two, see Example \ref{Ex:Boundary} and Figure 
\ref{Fig:AmoebaBoundaryApproximation}, and the computation of the Betti decomposition, see Example \ref{Exa:BettiDecomp} and Figure \ref{Fig:BettiDecompositon}.\\

The methods developed in this paper are essential for the comprehension and computation of
finitely many representatives of amoebas of ideals, which we call \textit{amoeba bases}, see Theorem
\ref{Thm:AmoebaBasesBoundary}. We call a set $G =
\{g_1,\ldots,g_s\}$ an \struc{\textit{amoeba basis}} for a finitely generated ideal
$I$ if $\langle g_1,\ldots,g_s\rangle= I$, $\cA(I) = \bigcap_{j = 1}^s
\cA(g_j)$ and no subset of $G$ satisfies these properties, see Definition
\ref{Def:AmoebaBasis}. This definition is analogous to Gr\"{o}bner bases for classical algebraic
varieties, e.g., \cite{Cox:Little:OShea}, and tropical bases for tropical varieties, e.g.,
\cite{Bogart:et:al,Hept:Theobald:2,Speyer:Sturmfels}, Note that in contrast to the classical and the tropical case,
amoeba bases are not well understood yet and even their existence is unclear in general.
Here, we provide amoeba bases for a first, non-trivial class of ideals, namely for amoebas
of ideals corresponding to full ranked systems of linear equations, see Theorem
\ref{Thm:AmoebaBasesLinearCase}. The latter theorem was obtained in joint work with Chris
Manon and we are thankful for his approval to publish it in this paper.  Meanwhile, Theorem
\ref{Thm:AmoebaBasesLinearCase} was generalized by Nisse in \cite{Nisse:AmoebaBasesZeroDimensional} after our article had been published on the arXiv.\\

This article is organized as follows. In Section \ref{Sec:Preliminaries} we introduce
notation and recall all relevant facts
about amoebas. In Section \ref{Sec:NecessaryCondition} we prove the
implication \eqref{Equ:MainTheorem} of Theorem \ref{Thm:Main1}. In Section \ref{Sec:GenericallySufficient} we show
that \eqref{Equ:MainTheorem} is in fact an equivalence if $\mathbf{w} \in
\R^n$ is a non-singular point of the contour. We also prove Corollary \ref{Cor:Main1} and
discuss its consequences. In Section \ref{Sec:Computation} we review approximation methods for amoebas and develop
an algorithm for computing the boundary and the contour. In Section
\ref{Sec:AmoebaBases} we introduce the concept of an amoeba basis, give an overview
about known facts, and show why a description of the boundary of amoebas is crucial for the computation of amoeba bases, see Theorem \ref{Thm:AmoebaBasesBoundary}.\\

We remark that some of the results of this article, mostly Section \ref{Sec:NecessaryCondition}, are part of the thesis of the second author  \cite{deWolff:Diss}.

\bigskip

\subsection*{Acknowledgments.} We would like to thank Hannah Markwig and Thorsten Theobald
for their support during the development of this article. We thank Taylor Brysiewicz, Aur\'{e}lien Greuet, Simon Hampe, Johannes Rau, Elias Tsigaridas, and particularly Laura Matusevich for their helpful comments.
We cordially thank Chris Manon for his contribution on amoeba bases, specifically regarding Theorem
\ref{Thm:AmoebaBasesLinearCase}, and his approval to present this result here.\\

The first author was partially supported by GIF Grant no.\ 1174/2011. The second author 
was partially supported by GIF Grant no.\ 1174/2011, DFG project MA 4797/3-2 and DFG 
project TH 1333/2-1.

\section{Preliminaries}
\label{Sec:Preliminaries}

\subsection{Amoebas}
\label{SubSec:Amoebas}
As follows we always consider irreducible polynomials with a smooth variety. The amoeba $\cA(f)\subset \R^n$ of an irreducible Laurent polynomial $f$ is a full-dimensional real analytic set with open, convex components of the complement. Note that $\cA(f)$ is not necessarily full-dimensional if $f$ is reducible, see \cite[p. 58, Figure 2]{Rullgard:Diss}. Furthermore, every component of the complement of $\cA(f)$ corresponds to a lattice point $\alp$ in the Newton polytope $\New(f)$ of $f$ via the \struc{\textit{order map}}, see \cite{Forsberg:Passare:Tsikh}, which can be interpreted as a multivariate analog of the classical argument principle from complex analysis.
\begin{eqnarray}
  \label{Equ:Ordermap}
\struc{\ord}: \R^n \setminus \cA(f) & \to & \New(f) \cap \Z^n, \quad \mathbf{w} \mapsto (u_1,\ldots,u_n) \ \text{ with } \\
  u_j & = & \frac{1}{(2\pi i)^n} \int_{\Log|\mathbf{z}| = \mathbf{w}} \frac{z_j \partial_j f(\mathbf{z})}{f(\mathbf{z})}
    \frac{dz_1 \cdots dz_n}{z_1 \cdots z_n} \ \text{ for all } \ 1 \le j \le n \, .\nonumber
\end{eqnarray}
The order map is invariant on all points of a particular component of the complement of $\cA(f)$ and it is injective on the set of all components of the complement of $\cA(f)$; see \cite[Prop. 2.5, Theorem 2.8]{Forsberg:Passare:Tsikh}. We
denote the component of the complement of $\cA(f)$ containing all points of order $\alp \in \Z^n \cap
\New(f)$ by $E_\alp(f)$, i.e.,
\begin{eqnarray*}
 \struc{E_\alp(f)} & = & \{\mathbf{w} \in \R^n \setminus \cA(f) \ : \ \ord(\mathbf{w}) = \alp\}.
\end{eqnarray*}
Here, we treat $\alp \in \Z^n$ simultaneously as an exponent and as a lattice point in a polytope with slight abuse of notation.\\

In the following paragraph we introduce and motivate the \textit{extended boundary} of an amoeba. For a finite set $A \subset \Z^n$ we define the \struc{\textit{parameter space} $(\C^*)^A$} as the set of 
all Laurent polynomials $f = \sum_{\alp \in A} b_\alp \mathbf{z}^{\alp}$ with support set 
$A$ and non-vanishing complex coefficients $b_\alp \in \C^*$. Note that $\cV(f)$ is dilation invariant, i.e. $\cV(f) = \cV(c \cdot f)$ for all $c \in \C^*$. Thus, we assume that all $f \in (\C^*)^A$ are monic in the sense 
that the leading term in reverse lexicographic term 
ordering equals one. Moreover, every monomial is a unit 
in the corresponding Laurent polynomial ring and hence $\cV(f)$ and $\cA(f)$ are 
invariant under translations of support sets
\begin{eqnarray}
	\tau_\beta: \Z^n \to \Z^n, \quad \mathbf{z}^{\alp} \mapsto \mathbf{z}^{\alp +
\beta} \label{Equ:Translation}
\end{eqnarray}
with $\beta \in \Z^n$. Hence, we assume from now on that
\begin{eqnarray*}
	A \subset \N^n.
\end{eqnarray*}
 Thus, we can represent each Laurent polynomial by a monic, regular polynomial with support set $A \subset \N^n$ in $(\C^*)^A$. To indicate this representation we write in the following occasionally $f \in \C[\mathbf{z}]$ with $\cV(f) \subset (\C^*)^n$ with slight abuse of notation. 
 
 We can identify every polynomial in $(\C^*)^A$ with its coefficient vector. Therefore, we can identify $(\C^*)^A$ with a $(\C^*)^{d-1}$ space, where $d = \# A$. The Euclidean metric on  $(\C^*)^{d-1}$ induces a \struc{\textit{parameter metric}} $\struc{d^A}: (\C^*)^A \times (\C^*)^A \ra \R_{\geq 0}$ on $(\C^*)^A$ in the following way. Let $f = \sum_{\alp \in A} b_\alp \cdot \mathbf{z}^{\alp}$ and $g =
\sum_{\alp \in A} c_\alp \cdot \mathbf{z}^{\alp}$, then
\begin{eqnarray*}
	\struc{d^A(f,g)} & =	& \lf(\sum_{\alp \in A} \lf|b_\alp - c_\alp\ri|^2\ri)^{1/2}.
\end{eqnarray*} 
Note that $d^A$ is also dilation invariant since every polynomial in $(\C^*)^A$ is assumed to be monic.

\begin{defn}
Let $f \in (\C^*)^A$. We call a point $\mathbf{w} \in \cA(f)$ an \struc{\textit{extended boundary point}}
of $\cA(f)$ if there exists $\alp \in \New(f) \cap \Z^n$ such that for every $\eps >
0$ there exists $g \in \cB_{\eps}^A(f) \subset (\C^*)^A$ with $\mathbf{w} \in
E_{\alp}(g)$. Here, $\struc{\cB_{\eps}^A(f)}$ denotes the open ball around the polynomial $f$ with
radius $\eps$, with respect to the parameter metric $d^A$ on $(\C^*)^A$. We denote the extended boundary of $\cA(f)$ as $\struc{\partial^{e} \cA(f)}$.
\label{Def:ExtendedBoundary}
\endenvi
\end{defn}

This definition of the extended boundary yields a set which may exceed the topological boundary of the amoeba. In what follows we motivate this object. Namely, it guarantees that the boundary $\partial \cA(f)$ of an amoeba $\cA(f)$ is continuous under the change of coefficients.

More precisely, it is well-known that the number of components of the complement of $\cA(f)$ is lower semicontinuous under arbitrary small perturbations of coefficients, see \cite[Prop. 1.2]{Forsberg:Passare:Tsikh}. However, a new bounded component of the complement of can appear inside an amoeba after an arbitrary small change of the coefficients in $(\C^*)^A$, as the following example shows. Note that an amoeba $\cA(f)$ is called \struc{\textit{solid}} if the number of connected components of its complement is minimal, i.e., if every component of the complement of $\cA(f)$ corresponds to a vertex in $\New(f)$.

\begin{exa}
\label{Ex:Boundary}
Let $f = z_1^3 + z_2^3 + c z_1z_2 + 1$ with $c \in \R$. Then $\cA(f)$ is solid if and only if $c \in [-3,1]$. Particularly, for $c = 1$ every sufficiently small
neighborhood of the origin is contained in the interior of $\cA(f)$. If we, however,
increase the coefficient of $z_1z_2$ by an arbitrary small $\eps > 0$, then the origin is
contained in a bounded component of the complement of $\cA(f)$ of order $(1,1)$. See Figure \ref{Fig:AmoebaExa}, and
\cite{deWolff:Diss,Passare:Rullgard:Spine,Theobald:deWolff:Genus1} for further details.
\label{Exa:BoundaryDefinition}
\endenvi
\end{exa}

\begin{figure}
\ifpictures
\includegraphics[width=0.4\linewidth]{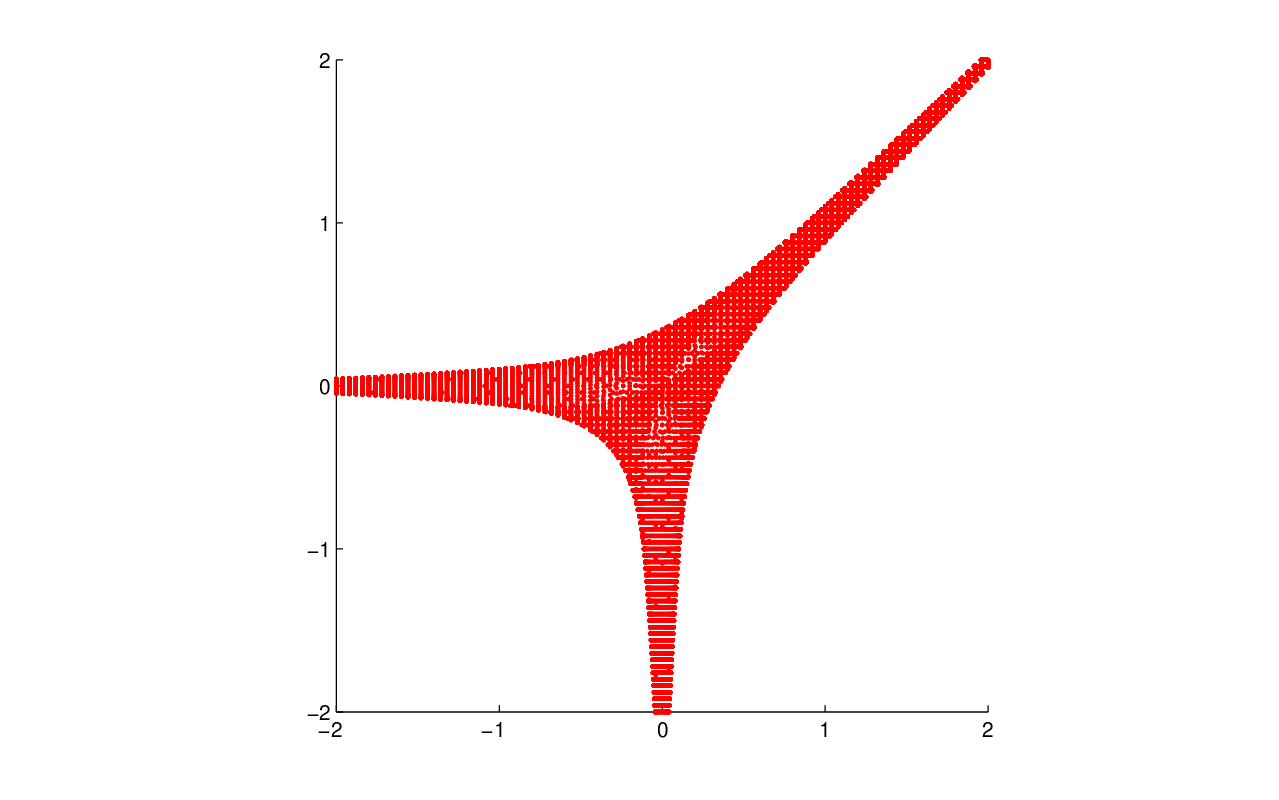}
\includegraphics[width=0.4\linewidth]{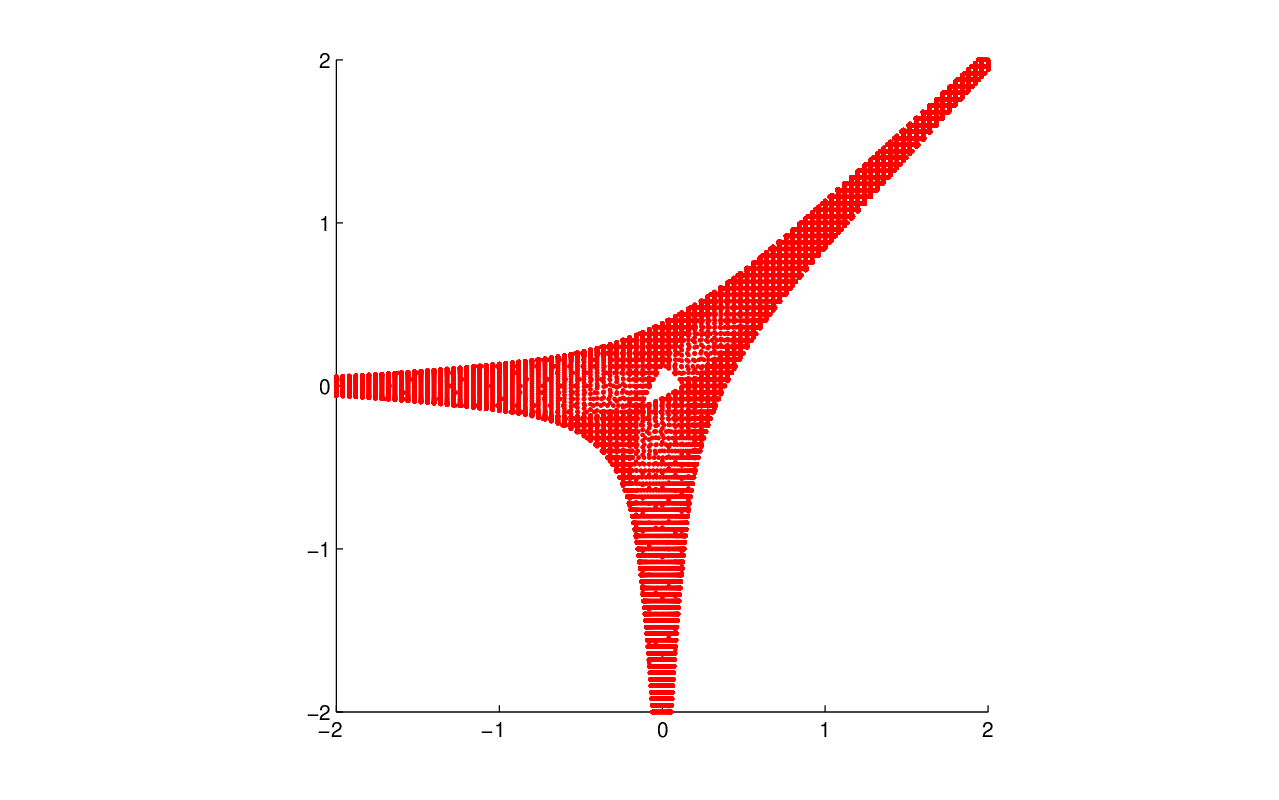}
\fi
\caption{The amoeba of $f = z_1^3 + z_2^3 + c z_1z_2 + 1$ with $c = 1$ and $c = 1.3$.}
\label{Fig:AmoebaExa}
\end{figure}

Hence, if we investigate the topological boundary of a family of amoebas given by 
polynomials $f$ as in the Example \ref{Exa:BoundaryDefinition} with coefficient $c$ being 
changed continuously from $1+\eps$ to $1$, then the boundary of the 
bounded component of the complement of $\cA(f)$ converges to the origin as a single, isolated point at the origin. 
But, obviously the origin is not part of the topological boundary of the amoeba of $f = 
z_1^3 + z_2^3 + z_1z_2 + 1$. Therefore, we consider the extended boundary, which 
characterizes such points as extended boundary points as well.\\

Recall that for every variety $\cV(f) \subset (\C^*)^n$ one defines its \struc{\textit{real locus}} as $\struc{\cV_{\R}(f)} =
\cV(f) \cap (\R^*)^n$. One essential tool to prove the main results of this article is to
investigate the following \struc{\textit{realification}} of polynomials. For every
polynomial $f \in \C[\mathbf{z}] = \C[z_1, \ldots,z_n]$ we denote its real and imaginary part as $\struc{f^{\re}}, \struc{f^{\im}} \in \struc{\R[\mathbf{x},\mathbf{y}]} = \R[x_1,\ldots,x_n,y_1,\ldots,y_n]$. So, we have
\[ 
  f(\mathbf{z}) \ = \ 
  f(\mathbf{x}+\text{i}\mathbf{y}) \ = \ f^{\re}(\mathbf{x},\mathbf{y}) + \text{i} \cdot
f^{\im}(\mathbf{x},\mathbf{y})\,,
\]
see also \cite{Theobald:deWolff:SOS}. With this notation $\cV(f) \subset (\C^*)^n \cong (\R^2 \setminus \{(0,0)\})^n \subset \R^{2n}$ coincides with the intersection of the real loci $\cV_{\R}(f^{\re}), \cV_{\R}(f^{\im}) \in (\R^2 \setminus \{(0,0)\})^n$ of the varieties $\cV(f^{\re}),\cV(f^{\im}) \subset (\C^*)^{2n}$ of $f^{\re}$ and $f^{\im}$.

Of $\cV(f)$ is smooth, then  $\cV_\R(f^{\re})$ 
and $\cV_\R(f^{\im})$ are also smooth after the embedding of $\cV(f)$ in $\R^{2n}$. We use this fact in Section \ref{Sec:NecessaryCondition}.

\subsection{Fibers}
\label{SubSec:Fibers}

Recall that a branch of the \struc{\textit{holomorphic logarithm}} is defined as
\begin{eqnarray*}
	\struc{\log_\C}: \C^* \to \C, \quad z \mapsto \log|z| + \text{i} \arg(z),
\end{eqnarray*}
where $\struc{\arg(z)}$ denotes the argument of the complex number $z$. Hence, the log absolute map $\log|\cdot|$ is the real part of any branch of the complex logarithm. Since the multivariate case works componentwise like the univariate case, the holomorphic logarithm $\struc{\Log_\C}$ yields a fiber bundle for $\Log|\cdot|$ such that the following diagram commutes \cite{deWolff:Diss,Mikhalkin:Annals,Mikhalkin:Survey}:
\begin{eqnarray}
\begin{xy}
	\xymatrix{
	(\C^*)^n \ar[rr]^{\Log_\C} \ar[rd]_{\Log|\cdot|} & & \R^n \times (S^1)^n \ar[ld]^{\RE}. \\
	& \R^n &
	}
\end{xy} \label{Equ:FiberBundle}
\end{eqnarray}

\noindent The \struc{\textit{fiber}} $\F_{\mathbf{w}}$ of each point $\mathbf{w} \in \R^n$ is given by
\begin{eqnarray*}
	\struc{\F_{\mathbf{w}}} & = & \{\mathbf{z} \in (\C^*)^n \ : \ \Log|\mathbf{z}| =
\mathbf{w}\}.
\end{eqnarray*}
Due to \eqref{Equ:FiberBundle} each $\F_{\mathbf{w}}$ is homeomorphic to an $n$-torus. For $f = \sum_{\alp \in A} b_\alp \mathbf{z}^{\alp}$ and $\mathbf{v} \in (\C^*)^n$ we define the \struc{\textit{fiber function}} \cite{deWolff:Diss,Theobald:deWolff:Genus1}
\begin{eqnarray*}
\struc{f^{|\mathbf{v}|}}: (S^1)^n \to \C, \quad  \boldsymbol{\phi} \mapsto
f(e^{\Log|\mathbf{v}| + \text{i} \boldsymbol{\phi}}) = \sum_{\alp \in A} b_\alp \cdot |\mathbf{v}|^\alp \cdot e^{\text{i} \langle \alp, \boldsymbol{\phi} \rangle},
\end{eqnarray*}
where $\struc{\langle \cdot, \cdot\rangle}$ denotes the standard inner product. The map $f^{|\mathbf{v}|}$ equals $f$ restricted to the fiber $\F_{\Log|\mathbf{v}|}$. More 
precisely, $f^{|\mathbf{v}|}$ is the function obtained by the pullback 
$\struc{\iota_{|\mathbf{v}|}^*(f)}$ of $f$ under the homeomorphism $\struc{\iota_{|\mathbf{v}|}}: (S^1)^n 
\to \F_{\Log|\mathbf{v}|} \subset (\C^*)^n$, which is induced by the fiber bundle 
structure. Hence, the corresponding zero sets satisfy
\begin{eqnarray}
	\cV(f^{|\mathbf{v}|}) \ = \ \cV(\iota_{|\mathbf{v}|}^*(f)) \ \cong \ \cV(f)
\cap \F_{\Log|\mathbf{v}|}, \label{Equ:FiberFunction1}
\end{eqnarray}
and thus
\begin{eqnarray}
	\Log|\mathbf{v}| \in \cA(f) & \Lera & \cV(f^{|\mathbf{v}|}) \neq
\emptyset. \label{Equ:FiberFunction2}
\end{eqnarray}

\bigskip

The $\Arg$ map is given by
\begin{eqnarray*}
	\struc{\Arg}: (\C^*)^n \to (S^1)^n, \quad (z_1,\ldots,z_n) \mapsto (\arg(z_1),\ldots,\arg(z_n)).
\end{eqnarray*}
It is a counterpart of the $\Log|\cdot|$ map, since it is given by the componentwise projection on the imaginary part of the multivariate complex logarithm $\Log_\C$.\\

It is easy to see that the $\Log|\cdot|$ and the $\Arg$ map including their fibrations extend to the realified version in $(\R^2 \setminus \{(0,0)\})^n$ of a given variety $\cV(f)$. Here, we denote these maps as $\Log_\R|\cdot|$ and $\Arg_\R$. Later, if the context is clear, then we just write $\Log|\cdot|$ and $\Arg$ with slight abuse of notation. Let $\struc{\cR}: (\C^*)^n \to (\R^2 \setminus \{(0,0)\})^n$ denote the \struc{\textit{realification homeomorphism}}. Then following diagram commutes:

\begin{eqnarray*}
\begin{xy}
	\xymatrix{
	\R^n \ar[r]^{\text{univ. covering}\phantom{X}} & (S^1)^n & \\
	  (\R^2 \setminus \{(0,0)\})^n \ar[u]^{\Arg_\R} \ar[rd]_{\Log_\R|\cdot|} &
(\C^*)^n \ar[l]_{\phantom{XXXx}\cR} \ar[r]^{\Log_\C\phantom{XX}}
\ar[d]_{\Log|\cdot|}
\ar[u]^{\Arg} & \R^n \times (S^1)^n
\ar[ld]^{\RE} \ar[lu]_{\IM}. \\
	& \R^n &
	}
\end{xy}
\end{eqnarray*}

\subsection{The Contour of an amoeba}
\label{SubSec:Contour}
Mikhalkin gave an explicit characterization of the \textit{contour} $\cC(f)$ of
the amoeba $\cA(f)$. For two smooth manifolds $\cM_1 \subset \R^m, \cM_2 \subset \R^n$ 
and a smooth map $g : \cM_1 \to \cM_2$ a point $\mathbf{v} \in \cM_1$ is
\struc{\textit{critical}} under $g$ if its Jacobian does not have maximal rank, i.e., 
$\min\{\dim~\cM_1, \dim~\cM_2\}$, at $\mathbf{v}$.
The \struc{\textit{contour} $\cC(f)$} of $\cA(f)$ is defined as the $\Log|\cdot|$ image of the set of 
critical points of $\cV(f)$ under the $\Log|\cdot|$ map. The contour of $\cA(f)$ is a 
closed real-analytic hypersurface in $\cA(f) \subset 
\R^n$; see \cite{Passare:Tsikh:Survey} and also Lemma \ref{Lem:Contour}. Note that $\cC(f)$ is not algebraic, since $\Log|\cdot|$ is not an algebraic map. It is easy to see  that $\partial \cA(f) \subset \cC(f)$, but in general the boundary $\partial \cA(f)$ is a 
proper subset of the contour $\cC(f)$ and hence $\cC(f)$ is not sufficient to describe the boundary, see Section \ref{Sec:NecessaryCondition} for further details. 

\subsection{The Gau\ss \ Map}
\label{SubSec:GaussMap}
For a smooth variety $\cV(f) \subset (\C^*)^n$, interpreted as a complex,
smooth $(n-1)$-manifold, the \struc{\textit{Gau\ss{} map}} is given by
\begin{eqnarray}
	\struc{G}: \cV(f) \ra \P^{n-1}_\C, \quad \mathbf{z}=(z_1,\ldots,z_n) \mapsto
\lf(\frac{\partial f}{\partial z_1}(\mathbf{z}):\ldots:\frac{\partial f}{\partial
z_n}(\mathbf{z})\ri). \label{Equ:GaussMap}
\end{eqnarray}

\noindent Geometrically, the Gau\ss{} map can be interpreted as follows: For every 
point $\mathbf{v} \in \cV(f)$ consider the \struc{\textit{tangent space} $T_{\mathbf{v}}\cV(f)$} in 
$\C^n$ of $\cV(f)$ at $\mathbf{v}$. Then the $i$-th entry of $G(\mathbf{v})$ is the $i$-th standard basis vector of the tangent plane $T_{\mathbf{v}}\cV(f)$. Hence, the image of the Gau\ss{} map is 
in bijection with the tangent bundle $T\cV(f)$ of 
$\cV(f)$.

Since $\cV(f) \subset (\C^*)^n$, and $(\C^*)^n$ is a Lie group, there exists a canonical trivialization of the tangent bundle $T 
\cV(f)$, i.e., $T \cV(f) \approx \cV(f) \times \C^{n-1}$.
Namely, there exists an isomorphism $T_{\mathbf{z}}\cV(f) \cong 
T_{\mathbf{1}}\cV(f)$, with $\struc{\mathbf{1}} = (1,\ldots,1)$, which is induced by the group
action on $(\C^*)^n$, which itself is given by the multiplication with $\mathbf{z}^{-1}$. 
Thus, if $\mathbf{v}$ is given, then we can always treat the affine space $T_{\mathbf{v}}\cV(f)$ as a linear space. For further details see \cite{Lee,Mikhalkin:Survey}.

For the logarithmic Gau\ss \ map $\gamma$, which we defined in the introduction, an analogous trivialization can be achieved in logarithmic coordinates.

\section{A Necessary Condition for Contour Points to be Boundary Points}
\label{Sec:NecessaryCondition}

The goal in this section is to prove Implication \eqref{Equ:MainTheorem} of Theorem
\ref{Thm:Main1}. We show that a point $\mathbf{w} \in \R^n$ is an extended
boundary point of an amoeba $\cA(f)$ only if \textit{every} point of the variety $\cV(f)$
intersected with the fiber $\F_{\mathbf{w}}$ is critical under the $\Log|\cdot|$ map,
i.e., by Mikhalkin's theorem, if it has real (projective) image under the logarithmic Gau\ss{} map. For points in the \textit{topological} boundary this statement is an easy consequence of the implicit function theorem. However, the statement is less obvious for points in the \textit{extended} boundary which do not belong to the topological boundary. Moreover, we need the techniques developed in this section to complete the proof of Theorem \ref{Thm:Main1} and to show further statements in Section \ref{Sec:GenericallySufficient}.

The key idea of our approach is to study how $f^{\re}$ and $f^{\im}$, the real and the imaginary parts of $f$, behave on a
particular fiber, and in particular, how their varieties intersect.\\

We briefly recall some of the known facts regarding boundaries of amoebas. Let $\struc{\mathbf{0}}$ denote the origin and let $\struc{e_1,\ldots,e_n}$ denote the denote the standard vectors in $\Z^n$. As mentioned in the introduction, the boundary of amoebas is well understood in the linear case. The following proposition provides a characterization. It was given by Forsberg, Passare and Tsikh \cite[Prop. 4.2]{Forsberg:Passare:Tsikh} and follows from a direct calculation.

\begin{prop}[Forsberg, Passare, Tsikh]
If $f = 1 + \sum_{j = 1}^n b_j z_j$ is linear, then  $\mathbf{v} \in \ovl{E_{e_j}(f)} \cap \partial \cA(f)$ if and only if $|b_j
v_j| = 1 + \sum_{k \in \{1,\ldots,n\} \setminus \{j\}} |b_k v_k|$. Analogously, $\mathbf{v} \in \ovl{E_{\mathbf{0}}(f)} \cap \partial 
\cA(f)$ if and only if $\sum_{j = 1}^n |b_j v_j| = 1$.
\label{Prop:LinearCase}
\end{prop}

If $n=2$, then the extended boundary and
the contour of an amoeba coincide for a far more general class due to a result by Mikhalkin and Rullg{\aa}rd
\cite{Mikhalkin:Rullgard}. Remember that a variety $\cV(f)$ is
called \struc{\textit{real}} if it is invariant under complex conjugation of the variables
$z_1,\ldots,z_n$. For any lattice polytope $P \subset \Z^n$ we denote the \struc{\textit{volume}} of $P$ 
by the \struc{\textit{volume form} $\vol(P)$} of $\Z^n$ satisfying:
\begin{enumerate}
 \item The standard simplex has volume one.
 \item $\vol(\cdot)$ is under $\SL_n(\Z)$ actions and affine translations.
 \item For two polytopes $P$ and $Q$ which have an intersection of codimension at least one it holds $\vol(P \cup Q) = \vol(P) + \vol(Q)$.
\end{enumerate}
For additional information see \cite[p. 182 et seq.]{Gelfand:Kapranov:Zelevinsky}. Furthermore, one 
denotes the \struc{\textit{area}} of an amoeba of a complex polynomial $f \in \C[z_1,z_2]$ with respect to 
the Lebesgue measure as $\struc{\area(\cA(f))}$. $\cV(f) \subset (\C^*)^2$ is called
\struc{\textit{real up to scalar multiplication}} if there exist $a,b_1,b_2 \in \C^*$ such that 
$a f(z_1 / b_1,z_2 / b_2)$ has real coefficients. Mikhalkin and Rullg{\aa}rd showed the following statement; see \cite[Theorem 1]{Mikhalkin:Rullgard}.

\begin{prop}[Mikhalkin, Rullg{\aa}rd]
Let $f \in \C[z_1,z_2]$ with $\vol(\New(f)) > 0$. Then the following conditions are equivalent
\begin{enumerate}
	\item $\area(\cA(f)) = \pi^2 \cdot \vol(\New(f))$
	\item $\Log|_{\cV(f)}$ is at most $2$ to $1$ and $\cV(f)$ is real up to scalar
multiplication.
	\item $\cV(f)$ is real up to scalar multiplication and its real locus
$\cV_{\R}(f)$ is a (possibly singular) Harnack curve.
\end{enumerate}
Furthermore, if any of these conditions is satisfied, then $\Log|\cV_\R(f)| = \partial \cA(f)$.
\label{Prop:Harnack}
\end{prop}

The boundary and the contour of an amoeba do not coincide in general. Later, we see that the boundary and the contour for the amoeba introduced in Example \ref{Ex:Boundary} do not coincide for positive coefficients of the term $z_1z_2$, see Figure \ref{Fig:AmoebaBoundaryApproximation}. Further examples can be found in \cite{Mikhalkin:Annals,Passare:Tsikh:Survey,Theobald:ComputingAmoebas}. \\

In what follows we prove the first part of Theorem \ref{Thm:Main1}. Let
$\mathbf{v} \in (\C^*)^n$ with $\Log|\mathbf{v}| = \mathbf{w}$. After the
realification of $f$ the fiber $\F_{\mathbf{w}}$ is given by $\{(\mathbf{x},\mathbf{y})
\in \R^{2n} \ : \ x_j^2 + y_j^2 = |v_j|^2 \text{ for } 1 \leq j \leq n\}$, i.e.,
every fiber is a real variety in $\R^{2n}$ of codimension $n$. By construction we have
$\Fa[|\mathbf{v}|,\re] = f^{\re}_{|\F_{\mathbf{w}}}$, and analogously for the imaginary part. 
Thus, it does not matter whether we investigate the real part $\Fa[|\mathbf{v}|,\re]$ of the 
fiber function $\Fa[|\mathbf{v}|]$, given by the restriction of $f$ to $\F_{\mathbf{w}}$ 
or whether we first take the real part $f^{\re}$ of $f$ and restrict it afterwards to the fiber
$\F_{\mathbf{w}}$, since the bundle structure is preserved under the realification (see Section \ref{SubSec:Fibers}). Hence,
we have
\begin{eqnarray*}
	\cV(\Fa[|\mathbf{v}|,\re]) & \cong & \cV(f^{\re}) \cap \F_{\mathbf{w}},
\end{eqnarray*}
and analogously for the imaginary part. The following lemma describes the structure of
$\cV(\Fa[|\mathbf{v}|,\re])$.

\begin{lemma}
Let $f \in \C[\mathbf{z}]$ with $\cV(f) \subset (\C^*)^n$ non-singular and
$\Log|\mathbf{v}| = \mathbf{w} \in \R^n$ with $\F_{\mathbf{w}} \cap \cV(f) \neq
\emptyset$. Then generically $\cV(\Fa[|\mathbf{v}|,\re])$ and $\cV(\Fa[|\mathbf{v}|,\im])$ are
real smooth $(n-1)$-manifolds.
\label{Lem:DimFiberManifolds}
\end{lemma}

 Note that $\cV(\Fa[|\mathbf{z}|,\re])$ and $\cV(\Fa[|\mathbf{z}|,\im])$ are in
general neither connected nor smooth. In the lemma, the term ``\struc{\textit{generically}}'' means 
that if $\cV(\Fa[|\mathbf{v}|,\re])$ or $\cV(\Fa[|\mathbf{v}|,\im])$ is singular or their 
intersection has dimension lower than $n-1$, then the subset of all polynomials $g$ in 
the parameter space $(\C^*)^A$ with the same property and which are located in a
neighborhood of $f$ (with respect to the coefficient metric $d_A$) has codimension at 
least one. 

\begin{proof}(Lemma \ref{Lem:DimFiberManifolds})
We show only the real case, when $\cV(f^{\re})$ is a real, non-singular
$(2n-1)$-manifold in $\R^{2n}$. The fiber $\F_{\mathbf{w}}$ is
given by $n$ real, non-singular hypersurfaces in $\R^{2n}$ defined by
$x_j^2 + y_j^2 = |v_j|^2$, $1\leq j\leq n$. Since $\F_{\mathbf{w}} \cap \cV(f^{\re})
\neq \emptyset$ by assumption, it is generically the transverse, i.e., non-singular intersection of $n+1$
real, non-singular $(2n-1)$-manifolds, so it is a real,
generically non-singular manifold of dimension $(n+1)\cdot(2n-1)-n\cdot(2n)=n-1$.
\end{proof}

\begin{lemma}
Let $f \in \C[\mathbf{z}]$ with $\cV(f) \subset (\C^*)^n$. Then a point
$\mathbf{v} \in \cV(f)$ is critical under the $\Log|\cdot|$ map if and only if it is critical
under the $\Arg$ map.
\label{Lem:CriticalPointsLogArg}
\end{lemma}

 This statement follows already at least implicitly from Mikhalkin's
argumentation on the logarithmic Gau\ss{} map \cite{Mikhalkin:Annals} and was also
observed by Nisse and Passare \cite{Nisse:LogGauss} before. For convenience, we give an independent proof in the Appendix \ref{Sec:Appendix}. 

We denote the \struc{\textit{tangent space}} at a smooth point $\mathbf{v}$ of a manifold $\cM$ as $\struc{T_{\mathbf{v}}\cM}$ and its \struc{\textit{orthogonal complement}} in the ambient space of $\cM$ as $\struc{(T_{\mathbf{v}}\cM)^\perp}$.

\begin{lemma}
Let $f \in \C[\mathbf{z}]$ with $\cV(f) \subset (\C^*)^n$ smooth and let
$\mathbf{v} \in \cV(f)$ be a non-critical point under the $\Arg$ map with
$\Arg(\mathbf{v}) = \boldsymbol{\phi} \in [0,2\pi)^n$. Then $T_{\boldsymbol{\phi}}\cV(\Fa[|\mathbf{v}|,\re]) \neq
T_{\boldsymbol{\phi}}\cV(\Fa[|\mathbf{v}|,\im])$ and furthermore both $\cV(\Fa[|\mathbf{v}|,\re])$ and 
$\cV(\Fa[|\mathbf{v}|,\im])$ are smooth at $\boldsymbol{\phi}$.
\label{Lem:RegularPointsonFiber}
\end{lemma}

 Note that $T_{\boldsymbol{\phi}}\cV(\Fa[|\mathbf{v}|,\re])$ is the tangent space of the variety $\cV(\Fa[|\mathbf{v}|,\re])$ of the real part of the fiber function $\Fa[|\mathbf{v}|]$ at $\boldsymbol{\phi}$ (analogously for the imaginary part). So,
$T_{\boldsymbol{\phi}}\cV(\Fa[|\mathbf{v}|,\re])$ is a subset of real dimension $n-1$ of $(S^1)^n$ (isomorphic to the fiber
$\F_{\mathbf{w}}$), and hence of codimension one if $\cV(\Fa[|\mathbf{v}|,\re])$ is smooth at $\boldsymbol{\phi}$.

\begin{proof}
Assume first $T_{\boldsymbol{\phi}}\cV(\Fa[|\mathbf{v}|,\re])
= T_{\boldsymbol{\phi}}\cV(\Fa[|\mathbf{v}|,\im])$ and
$\cV(\Fa[|\mathbf{v}|,\re]),\cV(\Fa[|\mathbf{v}|,\im])$ are smooth at
$\boldsymbol{\phi}$. Since
$T_{\boldsymbol{\phi}}\cV(\Fa[|\mathbf{v}|]) =
T_{\boldsymbol{\phi}}\cV(\Fa[|\mathbf{v}|,\re]) \cap
T_{\boldsymbol{\phi}}\cV(\Fa[|\mathbf{v}|,\im])$ and
$\dim~\cV(\Fa[|\mathbf{v}|,\re]) = \dim~\cV(\Fa[|\mathbf{v}|,\im]) = n-1$ by Lemma
\ref{Lem:DimFiberManifolds}, we have
$\dim~T_{\boldsymbol{\phi}}\cV(\Fa[|\mathbf{v}|]) = n-1$. Since
$\cV(\Fa[|\mathbf{v}|]) \cong \cV(f) \cap \F_{\mathbf{w}}$ for $\mathbf{w} =
\Log|\mathbf{v}|$ and hence, $T_{\boldsymbol{\phi}}\cV(\Fa[|\mathbf{v}|]) \cong
T_{\boldsymbol{\phi}}\cV(f) \cap \F_{\mathbf{w}}$, there is an immersion of an
$(n-1)$-dimensional subspace of $T_{\mathbf{v}}\cV(f)$
into $\Arg(T_{\mathbf{v}}\cV(f))$, which yields that $\mathbf{v}$ is critical by the same argument as in Lemma \ref{Lem:CriticalPointsLogArg}.

Assume $\cV(\Fa[|\mathbf{v}|,\re])$ is singular at $\boldsymbol{\phi}$. Then
$T_{\boldsymbol{\phi}}\cV(\Fa[|\mathbf{v}|,\re]) = [0,2\pi)^n$ and hence,
$T_{\boldsymbol{\phi}}\cV(\Fa[|\mathbf{v}|]) =
T_{\boldsymbol{\phi}}\cV(\Fa[|\mathbf{v}|,\im])$. Since $\cV(f)$ is smooth,
$\cV(\Fa[|\mathbf{v}|,\im])$  may not be singular at $\boldsymbol{\phi}$ either. Therefore, $\dim(T_{\boldsymbol{\phi}}\cV(\Fa[|\mathbf{v}|])) = n-1$. The rest works in the same way as above.
\end{proof}

 We are now ready to prove the main result of this section, namely the
first part of Theorem \ref{Thm:Main1} given by the Implication \eqref{Equ:MainTheorem}.
The idea of proof is that if there exists a point $\mathbf{v} \in \cV(f) \cap
\F_{\mathbf{w}}$ with $\mathbf{v} \notin S(f)$, then the manifolds
$\cV(\Fa[|\mathbf{v}|,\re])$ and $\cV(\Fa[|\mathbf{v}|,\im])$ intersect regularly in
$\Arg(\mathbf{v})$. But this means that they also intersect for every small perturbation
of the coefficients of $f$. This is a contradiction to the assumption $\mathbf{w}
\in \partial^{e} \cA(f)$, which means exists a small perturbation of the coefficients
yielding $\cV(f) \cap \F_{\mathbf{w}} = \emptyset$, i.e., $\cV(\Fa[|\mathbf{v}|,\re])
\cap \cV(\Fa[|\mathbf{v}|,\im]) = \emptyset$. 

\begin{proof}\textit{(Theorem \ref{Thm:Main1}, Implication \eqref{Equ:MainTheorem})}
Let $\mathbf{w} \in \partial^{e} \cA(f)$ and assume that there exists $\mathbf{v} \in \cV(f)
\cap \F_{\mathbf{w}}$ with $\mathbf{v} \notin S(f)$. By Theorem
\ref{Thm:MikhalkinContour}, $\mathbf{v}$ is a non-critical point under
the $\Log|\cdot|$ map and hence, by Lemma \ref{Lem:CriticalPointsLogArg}, $\mathbf{v}$ is a non-critical point under the $\Arg$ map, too. Thus, by Lemma \ref{Lem:RegularPointsonFiber} $T_{\boldsymbol{\phi}}\cV(\Fa[|\mathbf{v}|,\re]) \neq T_{\boldsymbol{\phi}}\cV(\Fa[|\mathbf{v}|,\im])$, and both $\cV(\Fa[|\mathbf{v}|,\re])$ and $\cV(\Fa[|\mathbf{v}|,\im])$ are regular at
$\boldsymbol{\phi}$, and thus also in a neighborhood $U_{\boldsymbol{\phi}} \subset (S^1)^n \cong \F_{\mathbf{w}}$. Therefore,
$\cV(\Fa[|\mathbf{v}|,\re])$ and $\cV(\Fa[|\mathbf{v}|,\im])$ intersect regularly in $U_{\boldsymbol{\phi}}$. Hence, there
exists a $\delta > 0$ such that in $U_\phi$ the intersection of every $\delta$-perturbation of
$\cV(\Fa[|\mathbf{v}|,\re])$ and $\cV(\Fa[|\mathbf{v}|,\im])$ is not empty.

Since $\mathbf{w} \in \partial^{e} \cA(f)$ we find some $g \in \cB_{\eps}^A(f) \subset (\C^*)^A$ for
every arbitrary small $\eps > 0$ such that $\mathbf{w} \notin \cA(g)$, i.e., $\cV(g) \cap
\F_{\mathbf{w}} = \emptyset$, and thus in particular $\cV(g^{|\mathbf{v}|,\re})
\cap \cV(g^{|\mathbf{v}|,\im}) = \emptyset$. The functions $f^{\re}, f^{\im}$ are continuous 
under a change of the coefficients of $f$. Hence, also
the regular loci of $\cV(f^{\re}), \cV(f^{\im})$ are continuous under a sufficiently small change the coefficients of $f$.
Therefore, by definition of the fiber function, $\cV(g^{|\mathbf{v}|,\re})$ and
$\cV(g^{|\mathbf{v}|,\im})$ are arbitrary small perturbations of
$\cV(\Fa[|\mathbf{v}|,\re])$ and $\cV(\Fa[|\mathbf{v}|,\im])$ in a neighborhood of the regular point $\boldsymbol{\phi}$. Thus,
$\cV(\Fa[|\mathbf{v}|,\re])$ and $\cV(\Fa[|\mathbf{v}|,\im])$ may not intersect
regularly in $\boldsymbol{\phi}$. This is a contradiction.
\end{proof}

\begin{figure}[ht]
\ifpictures
	\includegraphics[width=0.4\linewidth]{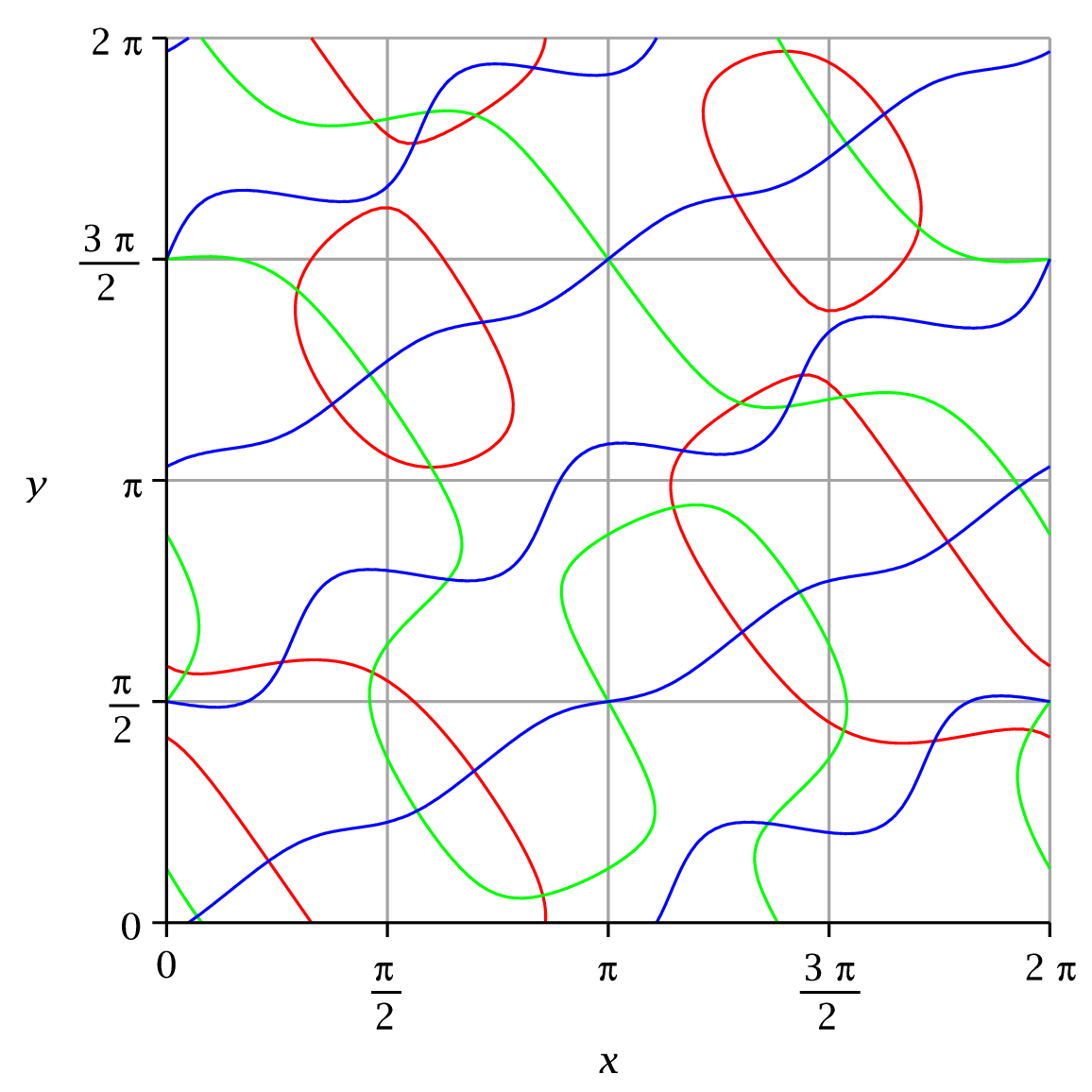}
	\includegraphics[width=0.4\linewidth]{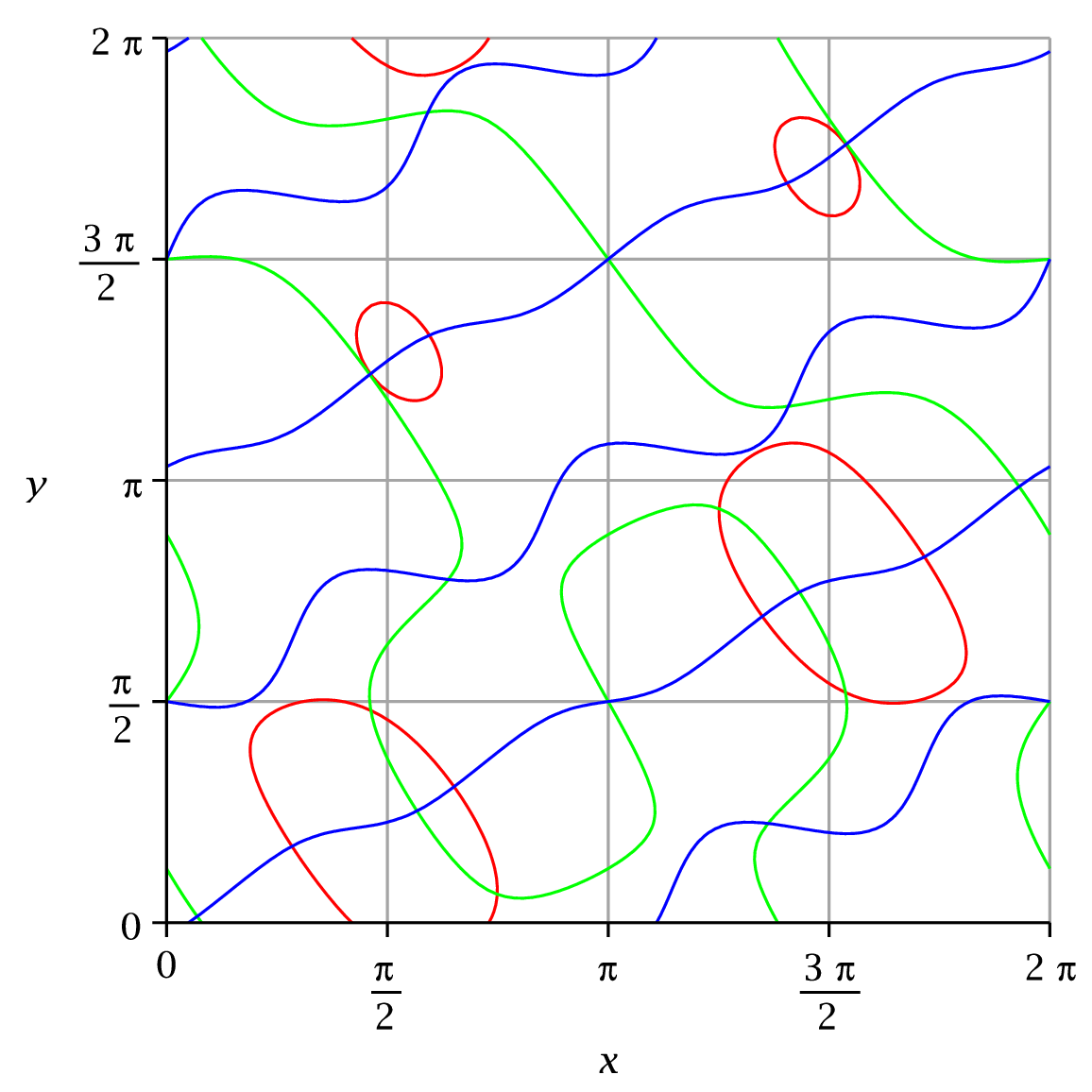}
	\includegraphics[width=0.4\linewidth]{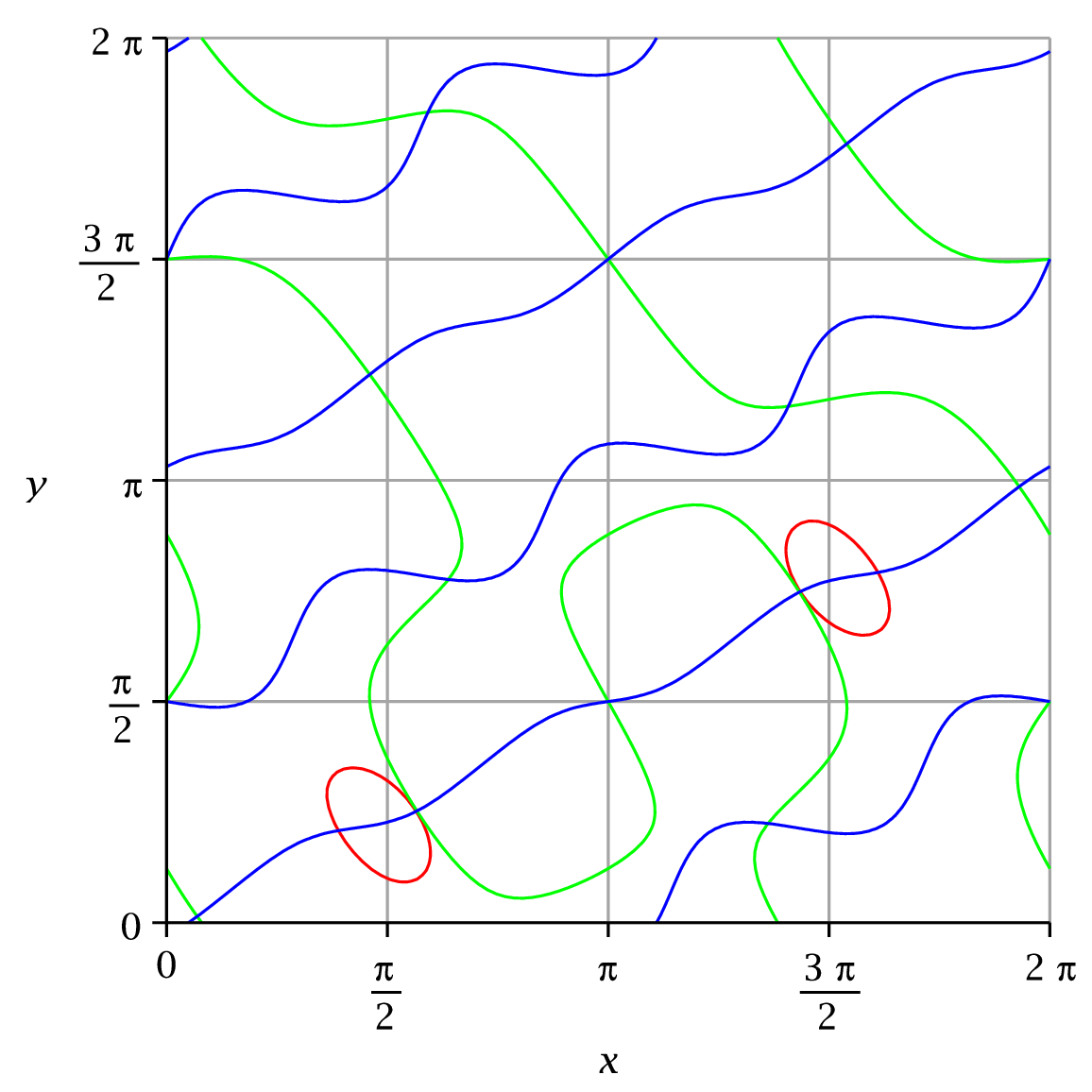}
	\includegraphics[width=0.4\linewidth]{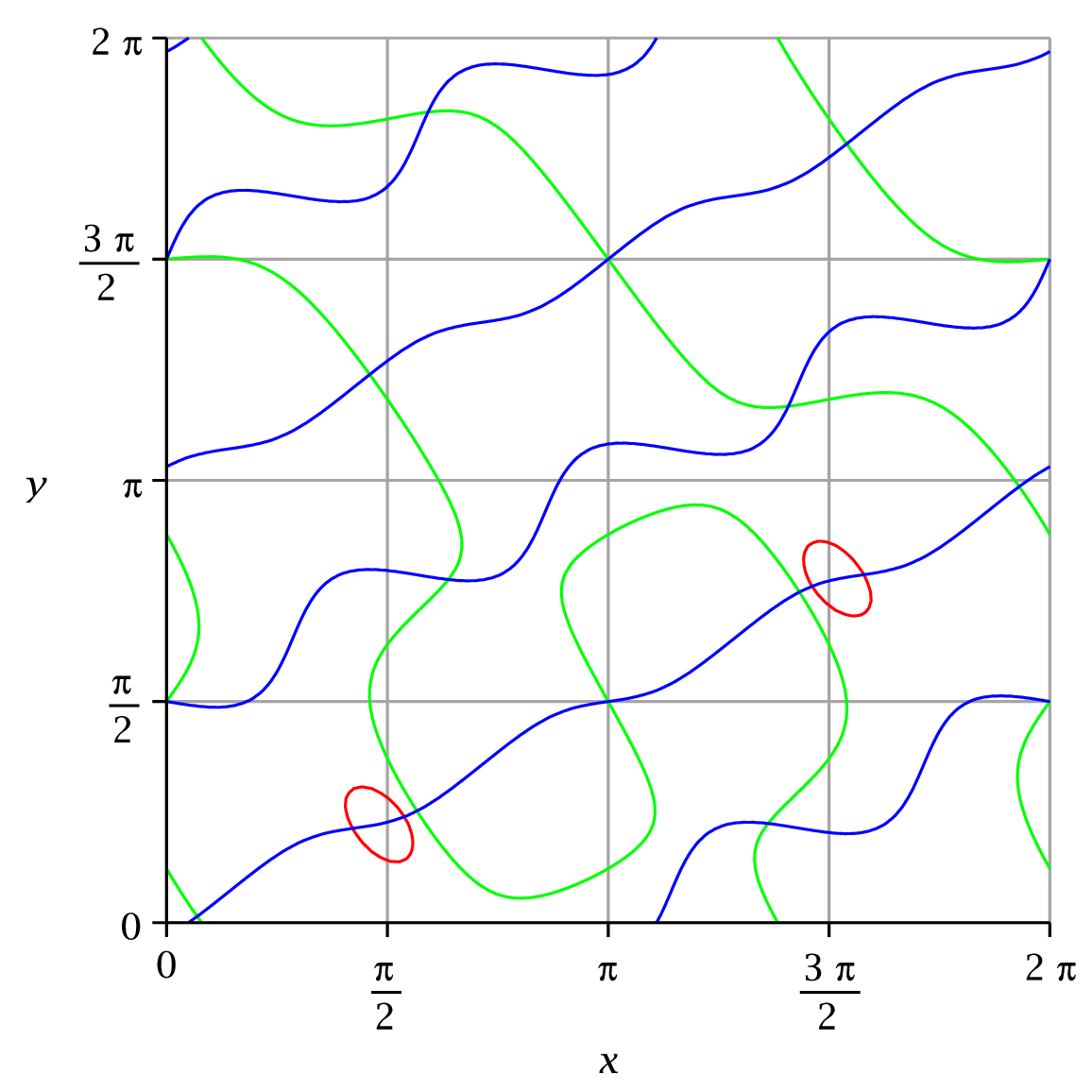}
\fi
	\caption{The behavior of $f = -2 z_1^2 - 2 z_1z_2^2 + 1,5 e^{\text{i} \pi \cdot
0,5}
z_1^{-1}z_2^{-1} + c$ on the fiber $\F_{(0,0)}$ for $c \in \{-1.2,-2.7,-4.6,-4.9\}$.}
	\label{Fig:BoundaryFiber}
\end{figure}

\noindent We finish this section with an example. For simplicity, we do not distinguish here between a fiber 
$\F_{\Log|\mathbf{v}|}$ and the corresponding torus $(S^1)^n$ of the fiber function 
$f^{|\mathbf{v}|}: (S^1)^n \to (\C^*)^n$ by slight abuse of notation.

\begin{exa}
Let $f$ be a Laurent polynomial given by 
\begin{eqnarray*}
	  f & = & -2 z_1^2 - 2 z_1z_2^2 + 1,5 e^{\text{i} \pi \cdot 0.5} z_1^{-1}z_2^{-1}
+ c.
\end{eqnarray*}
Consider the fiber $\F_{(0,0)}$ of the point $\Log|(1,1)|$ for $c = -1.2, -2.7,
-4.6$ and $-4.9$ shown in Figure \ref{Fig:BoundaryFiber}. In all pictures the red curve
corresponds to $\cV(\Fa[|(1,1)|,\re])$ and the green curve corresponds to
$\cV(\Fa[|(1,1)|,\im])$ in $\F_{(0,0)}$. Hence, the points in the intersection of the red and the green
curve are the points where the real and the imaginary part of $\Fa[|(1,1)|]$ vanish,
i.e., these are the intersection points of the fiber $\F_{(0,0)}$ with the variety
$\cV(f)$.

The blue curve presents the argument of points on the complex unit sphere, which are
critical points under the logarithmic Gau\ss{} map, i.e., the critical points of $\gamma$
on the fiber $\F_{(0,0)}$. Thus, by Corollary \ref{Cor:MikhalkinBoundary}, $(0,0)$ is part of the contour if there is a point where
the red, green and blue curve intersect and, by Theorem \ref{Thm:Main1}, $(0,0)$ may only be a boundary point if all intersection points of the red and the green
curve also intersect the blue one. Note that in this example a change of the coefficient $c$ along the real axis only changes the red curve.

Furthermore, observe for $c = -1.2$ in the upper left of Figure
\ref{Fig:BoundaryFiber} the red and green curve intersect regularly in several points and
hence in this case $(0,0) \in \cA(f)$.  On the upper right pictures with $c = -2.7$, there are two intersection points where all three curves intersect. But there
are other points where (only) the red and the green curve intersect regularly. Thus, in
this case $(0,0)$ is part of the contour but still $(0,0) \in \cA(f) \setminus \partial^{e}
\cA(f)$. On the lower left picture with $c =-4.6$ the only two intersection
points of the red and the green curve also intersect the blue one. Hence, in this case
$(0,0)$ might be part of the extended boundary. On the lower right picture with $c
=-4.9$, the red and the green curve do not intersect in any point anymore. Therefore, we
have $(0,0) \notin \cA(f)$.
\endenvi
\end{exa}

Note that the values $c = -2.7$ and $c = -4.6$ are not the exact values of $c$ such that the 
point $(0,0)$ lies in the contour or in the extended boundary of $\cA(f)$. They
are approximations of the exact values in order to visualize the situation.

\section{Characterization of the Extended Boundary}
\label{Sec:GenericallySufficient}

In the first part of this section we finish the proof of Theorem \ref{Thm:Main1}. Afterwards, 
we point out that the approach used to prove Theorem \ref{Thm:Main1} allows an alternative 
description of amoebas, their contour and their extended boundary by using the real locus of a 
variety of a particular ideal instead of critical points of the $\Log|\cdot |$ map and the 
logarithmic Gau\ss \ map. We formulate these results in Theorem 
\ref{Thm:BoundaryArbitraryDimension}. This alternative approach also yields a 
decomposition of the ambient space of an amoeba, which is given by the $0$-th Betti number of 
the real locus of the variety of the ideal mentioned above. In Theorem 
\ref{Thm:BettiDecompositionContour} we show that the contour and the extended boundary of an amoeba arise as certain intersections of cells in this decomposition.\\

To complete the proof of Theorem \ref{Thm:Main1} we need some facts. First, we characterize the critical locus of a hypersurface $\cV(f)$ with respect to the $\Log|\cdot|$ map. The following statement is known in the community, but unfortunately we did not find a reference.

\begin{lemma}
Let $f \in \C[\mathbf{z}]$ with $\cV(f) \subset (\C^*)^n$. Then the critical
locus $S(f) \subset \cV(f)$ is a real algebraic $(n-1)$-variety in $\R^{2n}$.
\label{Lem:CritPtsVar}
\end{lemma}

\begin{proof}
 By Theorem \ref{Thm:MikhalkinContour} the critical locus $S(f)$ of $f$ is given by all points in $\cV(f)$ which have projective real image under the logarithmic Gau\ss \, map.
Hence, a point $\mathbf{v} \in \R^{2n}$ is contained in $S(f)$ if and only if $f^{\re}(\mathbf{v}) = f^{\im}(\mathbf{v}) = 0$ and for all $j,k \in \{1,\ldots,n\}$ with $j \neq k$ it holds that
\begin{eqnarray}
	\left(||v_j \frac{\partial f}{\partial z_j}(\mathbf{v})||^2 \cdot v_k \frac{\partial f}{\partial z_k}(\mathbf{v}) + ||v_k \frac{\partial f}{\partial z_k}(\mathbf{v})||^2 \cdot v_j \frac{\partial f}{\partial z_j}(\mathbf{v})\right) & \cdot & \\ \label{Equ:ProofCharacterizationCriticalLocus}
\left(||v_j \frac{\partial f}{\partial z_j}(\mathbf{v})||^2 \cdot v_k \frac{\partial f}{\partial z_k}(\mathbf{v}) - ||v_k
\frac{\partial f}{\partial z_k}(\mathbf{v})||^2 \cdot v_j \frac{\partial f}{\partial z_j}(\mathbf{v})\right) & = & 0. \nonumber
\end{eqnarray}
The latter means that the $j$-th and the $k$-th entry of the image of $\mathbf{v}$ under the logarithmic Gau\ss \ map only differ by a real scalar. Since after a realification of \eqref{Equ:ProofCharacterizationCriticalLocus} all equations are
given by polynomials in $\R[\mathbf{x},\mathbf{y}]$, the critical locus is a real algebraic variety. For every smooth point $\mathbf{v}$ of
$S(f)$  the tangent space $T_{\mathbf{v}}S(f)$ of the critical locus at the point $\mathbf{v}$ is given by the subset of $T_{\mathbf{v}}\cV(f)$ of $\cV(f)$ which keeps the argument of the logarithmic Gau\ss \ image invariant. Thus, the real dimension is $n-1$.
\end{proof}

The next lemma is a statement about the logarithmic Gau\ss \ map by Mikhalkin
\cite[Lemma 2]{Mikhalkin:Annals} (note that he uses a different volume form); see also \cite{Passare:Tsikh:Survey}.

\begin{lemma}
Let $A \subset \Z^n$ be finite, $f \in (\C^*)^A$ with $\cV(f) \subset (\C^*)^n$ and 
$\mathbf{t} \in \P_{\C}^{n-1}$. Let $\gamma$ denote the logarithmic Gau\ss \ map. Then 
$\gamma$ has degree 
$\vol(\New(f))$. More precisely, if the pair $(f,t)$ is generic, then the set
\begin{eqnarray*}
	N_{\mathbf{t}} & = & \{\mathbf{z} \in (\C^*)^n \ : \ \mathbf{z} \in \cV(f) \cap \gamma^{-1}(\mathbf{t})\},
\end{eqnarray*}
has cardinality $\vol(\New(f))$ and hence it is in particular finite.
\label{Lem:LogGaussFinite}
\end{lemma}

For convenience, we recall the proof here. 

\begin{proof}
Without loss of generality let $t_1,\ldots,t_s = 0$ and $t_{s+1},\ldots,t_n \neq 0$. Then $\mathbf{v} \in N_{\mathbf{t}}$ if and only if $\mathbf{v}$ satisfies the following system of equations
\begin{eqnarray*}
	& & f(\mathbf{v}) \ = \ 0, \ z_1 \frac{\partial f}{\partial z_1}(\mathbf{v}) \ = \ 0, \ \ldots, \ z_s \frac{\partial f}{\partial z_s}(\mathbf{v}) \ = \ 0, \\
	& & \frac{z_{s+1}}{t_{s+1}} \frac{\partial f}{\partial z_{s+1}}(\mathbf{v}) \ = \
 \cdots \ = \frac{z_n}{t_{n}}\ \frac{\partial f}{\partial z_n}(\mathbf{v}).
 \end{eqnarray*}
 This is a system of $n$ equations in $n$ variables. Since $\cV(f)$ and $\cA(f)$ are invariant under a translation of the support set $A$ of $f$, see \eqref{Equ:Translation}, we can assume that no term cancellation occurs under differentiation. Hence, all polynomials in the system have the same Newton polytope and the statement follows by Kouchnirenko's Theorem \cite[p. 201]{Gelfand:Kapranov:Zelevinsky}, which states that generically the number of solutions is $\vol(\New(f))$.
\end{proof}

We also recall the structure of the contour of an amoeba itself, see \cite{Passare:Tsikh:Survey}.

\begin{lemma}
Let $f \in \C[\mathbf{z}]$ with $\cV(f) \subset (\C^*)^n$. Then the contour $\cC(f)$ is a closed, real analytic hypersurface in $\cA(f) \subset \R^n$.
\label{Lem:Contour}
\end{lemma}

Note that this particular lemma justifies talking about tangent spaces, regular and 
singular points with respect to the contour of an amoeba. As a last step, before we can finish the 
proof of Theorem \ref{Thm:Main1}, we need to show how the critical locus $\cS(f)$, the contour $\cC(f)$, 
the Gau\ss \ map and the logarithmic Gau\ss \ map interact. The following proposition is stated partially or
implicitly in several papers, for example \cite{Mikhalkin:Annals,Mikhalkin:Survey} and \cite[p. 277]{Passare:Tsikh:Survey}. But to 
the best of our knowledge it has not been stated in this form before.

\begin{prop}
Let $f \in \C[\mathbf{z}]$ with $\cV(f) \subset (\C^*)^n$. For the regular locus of $S(f)$ 
and the regular locus of the contour $\cC(f)$ the following diagram commutes.
\begin{eqnarray}
\begin{xy}
	\xymatrix{
	S(f) \ar[rr]^{\Log|\cdot|} \ar[rd]_{\gamma} & & \cC(f) \ar[ld]^{G}, \\
	& \P_{\R}^{n-1} &
	}
\end{xy} \label{Equ:LogGaussDiagramm}
\end{eqnarray}
where $G$ denotes the Gau\ss \ map and $\gamma$ denotes the logarithmic Gau\ss \ map. In 
particular, for a regular point $\mathbf{w} \in \cC(f)$, all points in $\cV(f) 
\cap \F_{\mathbf{w}} \cap S(f)$ have the same image under the logarithmic Gau\ss \ map 
and therefore $\cV(f) \cap \F_{\mathbf{w}} \cap S(f)$ is generically finite.
\label{Prop:SameTangentSpace}
\end{prop}

\begin{proof}
Let $\mathbf{v} \in S(f) \subset \cV(f)$ smooth and $\mathbf{w} = \Log|\mathbf{v}|$ with $\mathbf{w}$ smooth in $\cC(f)$. Since $\cC(f)$
is a smooth real analytic hypersurface and $\mathbf{w}$ is regular in $\cC(f)$ the tangent space $T_{\mathbf{w}}\cC(f)$ is real $(n-1)$-dimensional. 

Let $(b_1,\ldots,b_{n-1}) \subset (\C^*)^n$ be a basis of the tangent space 
$T_{\mathbf{v}}\cV(f)$ and $\mathbf{t}$ the corresponding normal vector. On the one hand, we show in the proof of Lemma \ref{Lem:CriticalPointsLogArg} (see Appendix \ref{Sec:Appendix}) that $\Log|T_{\mathbf{v}}\cV(f)|$ is spanned by $\Log|b_1|,\ldots,\Log|b_{n-1}| \in \R^n$ and hence is always a real $(n-1)$-dimensional hyperplane with normal vector $\Log|\mathbf{t}|$. On the other hand, $\mathbf{v}$ is in 
$S(f)$. Thus, $\gamma(\mathbf{v}) = (\log|t_1|:\cdots:\log|t_n|) \in \P_{\R}^{n-1} \subset \P_{\C}^{n-1}$.
Furthermore, by Lemma \ref{Lem:CritPtsVar}, $S(f)$ is a real $(n-1)$-manifold and for every point in a neighborhood of $\mathbf{v}$ in $S(f)$ the logarithmic Gau\ss \ map needs to remain real. Therefore, it follows that $T_{\mathbf{v}}S(f)$ is also spanned by $c' |b_1|,\ldots,c' |b_{n-1}|$ for some $c' \in \C^*$. Hence, 
\begin{eqnarray*}
	T_{\mathbf{w}}\cC(f) & = & \Log|T_{\mathbf{v}}\cV(f)| \text{ up to a linear transformation}. 
\end{eqnarray*}
Thus, the diagram \eqref{Equ:LogGaussDiagramm} commutes since the normal vector of $T_{\mathbf{w}}(\cC(f))$ is $\Log|\mathbf{t}|$ and so $G(\mathbf{w}) = \gamma(\mathbf{v})$. This also shows that all points in $S(f) \cap \F_{\mathbf{w}}$ have the same image under the logarithmic Gau\ss \ map. The finiteness of the set $\cV(f) \cap \F_{\mathbf{w}} \cap S(f)$ follows from Mikhalkin's Lemma \ref{Lem:LogGaussFinite}.
\end{proof}

The following result is an immediate consequence of Lemma \ref{Lem:LogGaussFinite} and Proposition \ref{Prop:SameTangentSpace}.

\begin{cor}
Let $f \in \C[\mathbf{z}]$ with $\cV(f) \subset (\C^*)^n$ and 
$\{\mathbf{w}_1,\ldots,\mathbf{w}_k\} \subset \cC(f) \subset \R^n$ the set of points in 
the contour of $f$, which are mapped to a point $\lambda \in \P_\R^n$ under the Gau\ss \ map such that $\gamma^{-1}(\lam)$ is finite. Then we have
\begin{eqnarray*}
	\sum_{j = 1}^k \# (S(f) \cap \F_{\mathbf{w}_j}) & = & \vol(\New(f)).
\end{eqnarray*}
\label{Cor:LogGaussFormula}
\end{cor}

\begin{flushright}
 $\square$
\end{flushright}

Note that $\gamma^{-1}(\lam)$ is generically finite by Lemma \ref{Lem:LogGaussFinite}.\\

We are now ready to complete the proof of Theorem \ref{Thm:Main1}.

\begin{proof}\textit{(Equivalence statement of Theorem \ref{Thm:Main1})}

Let $f \in \C[\mathbf{z}]$ with $\cV(f) \subset (\C^*)^n$ and $\mathbf{w}
\in \cC(f)$ such that the contour $\cC(f)$ is smooth at $\mathbf{w}$. Assume
that for every $\mathbf{z} \in \cV(f) \cap \F_{\mathbf{w}}$ it holds that
$\mathbf{z} \in S(f)$. That is, the right hand side of \eqref{Equ:MainTheorem} in Theorem \ref{Thm:Main1} is satisfied. By
Proposition \ref{Prop:SameTangentSpace} the $\Log|\cdot|$-images of the complex normal
vectors of the tangent spaces of all $\mathbf{z} \in \cV(f) \cap \F_{\mathbf{w}}$
coincide. We modify the coefficients of $f$ in a $\eps$-neighborhood
$\cB^A_\eps(f)$ with respect to the parameter metric $d^A$ in such a way that the variety $\cV(f)$ is locally
translated in the negative real direction of the $\Log_\C$-image of the 
unique complex normal vector of all points in $\cV(f) \cap \F_{\mathbf{w}}$. 
This is possible since the smooth part of $\cV(f)$ is locally linear in the coefficients of $f$.
After this translation the critical points of $f$ in $\F_{\mathbf{w}}$
are not contained in the translated variety $\cV(\tilde f)$, where 
$\tilde f$ denotes the translated polynomial. Let $\mathbf{v} \in S(f)$ with $\Log|\mathbf{v}| = \mathbf{w}$. 
Since we have only modified the coefficients by an arbitrary small
$\eps > 0$ with respect to the coefficient metric $d_A$ and as
$f^{|\mathbf{v}|,\re}$ and $f^{|\mathbf{v}|,\im}$  are continuous under the change
of coefficients, no new intersection points of $\cV(f^{|\mathbf{v}|,\re})$ and $\cV(f^{|\mathbf{v}|,\im})$ can 
appear in the torus $(S^1)^n$ isomorphic to the fiber $\F_{\mathbf{w}}$. Thus, 
$\cV(\tilde f) \cap \F_{\mathbf{w}} = \emptyset$.
\end{proof}

In Theorem \ref{Thm:Main1} we required $\mathbf{w} \in \R^n$ to be a smooth point 
of the contour of $\cA(f)$. The necessity of this assumption is illustrated in the following example.

\begin{exa}
Let $f = z_1^3 + z_1^3 + z_1z_2 + 1$. Although $\cV(f)$ is smooth the contour 
$\cC(f)$ has a singular point at the origin given by three intersecting branches; see 
Figure \ref{Fig:AmoebaBoundaryApproximation} in Section \ref{Sec:Computation}. 
Considering a neighborhood of the origin we can conclude by a continuity argument
in the fiber over the origin that there exist three different normal directions of tangent 
spaces of $\cV(f^{|\mathbf{0}|, \re})$ and $\cV(f^{|\mathbf{0}|, \im})$. Hence, 
we cannot guarantee that we can find a particular direction in the corresponding parameter 
space such that all intersections in the fiber vanish. It could happen that 
every perturbation of coefficients which lets one intersection point of 
$\cV(f^{|\mathbf{0}|, \re})$ and $\cV(f^{|\mathbf{0}|, \im})$ vanish turns another one 
into a regular intersection. Indeed, in this particular case the origin \textit{is} a 
boundary point and a suitable shifting direction of the parameter space is given by 
increasing the $z_1z_2$ term. See Figure \ref{Fig:AmoebaBoundaryApproximation} again.
\endenvi
\end{exa}

As a next step we give a characterization of the extended boundary of $\cA(f)$ equivalent to Theorem \ref{Thm:Main1}, which is computationally important. The following theorem was already partially stated as Corollary \ref{Cor:Main1} in the introduction.

\begin{thm}
Let $n \geq 2$, $f \in \C[\mathbf{z}]$ with $\cV(f) \subset (\C^*)^n$. Let 
$\Log|\mathbf{v}| = \mathbf{w} \in \R^n$ a smooth point of the contour such that $\gamma^{-1}(G(\mathbf{w}))$ is finite. Let $I = \langle
f^{\re}, f^{\im}, x_1^2 + y_1^2 = |v_1|^2,\ldots, x_n^2 + y_n^2 = |v_n|^2 \rangle$, where
$f^{\re},f^{\im}$ are given by the realification $z_j = x_j + \textnormal{i} y_j$ of
all variables. Then $\cV(I)$ has dimension at most $n-2$ and
\begin{enumerate}
	\item $\mathbf{w} \notin \cA(f)$ if and only if $\cV_{\R}(I) = \emptyset$,
	\item $\mathbf{w} \in \cC(f)$ if and only if 
	\begin{enumerate}
		\item $\cV_{\R}(I)$ contains roots of multiplicity greater than one for $n = 2$,
		\item $\cV_{\R}(I)$ contains an isolated point for $n \geq 3$.
	\end{enumerate}
	\item $\mathbf{w} \in \partial^{e} \cA(f)$ if and only if 
	\begin{enumerate}
		\item every root in $\cV_{\R}(I)$ has multiplicity greater than one for $n = 2$,
		\item $\cV_{\R}(I)$ is finite for $n \geq 3$.
	\end{enumerate}
	\item $\mathbf{w} \in \cA(f) \setminus \cC(f)$ else.
\end{enumerate}
\label{Thm:BoundaryArbitraryDimension} 
\end{thm}

Note that the assumptions on $\mathbf{w}$ are very mild. Both the singular points and the points with $\gamma^{-1}(G(\mathbf{w}))$ non-finite form a subset of the contour with codimension at least one by Lemma \ref{Lem:LogGaussFinite}, Lemma \ref{Lem:Contour} and Proposition \ref{Prop:SameTangentSpace}. Therefore, we can think of $\mathbf{w}$ as a \struc{\textit{generic}} point of the contour.

\begin{proof}
Note that $\cV(I)$ has dimension $n-2$.
Claim (1) follows since $\cV(f) \cap \F_{\mathbf{w}} \cong \cV(f^{|\mathbf{v}|,\re}) \cap 
\cV(f^{|\mathbf{v}|, \im}) = \cV(f^{\re}_{|\F_{\mathbf{w}}}) \cap 
\cV(f^{\im}_{|\F_{\mathbf{w}}})$ and $\F_{\mathbf{w}}$ is isomorphic to the real locus of 
the variety of the polynomials $x_1^2 + y_1^2 - |v_1|^2, \ldots, x_n^2 + y_n^2 - |v_n|^2$. 
We conclude $\cV_{\R}(I) \cong \cV(f) \cap \F_{\mathbf{w}}$ by construction. 
By Lemma \ref{Lem:DimFiberManifolds} $\cV(f^{|\mathbf{v}|, \re})$ and 
$\cV(f^{|\mathbf{v}|,\im})$ are real hypersurfaces in a real $n$-dimensional 
variety. The point $\mathbf{w}$ is contained in the contour of $\cA(f)$ if and only if $\cV_{\R}(I) 
\cong \cV(f) \cap \F_{\mathbf{w}}$ contains a critical point $\mathbf{v}$ of the 
$\Log|\cdot|$ map. Let $\boldsymbol{\phi} = \Arg(\mathbf{v})$ (to simplify the notation, we do not distinguish between 
$\boldsymbol{\phi}$ in complex space and after realification; see Section 
\ref{Sec:Preliminaries}). By Lemma \ref{Lem:CriticalPointsLogArg} and Lemma 
\ref{Lem:RegularPointsonFiber} this is the case if and only if 
$T_{\boldsymbol{\phi}}(\cV(f^{|\mathbf{v}|,\re})) = 
T_{\boldsymbol{\phi}}(\cV(f^{|\mathbf{v}|, \im}))$, i.e., the manifolds 
$\cV(f^{|\mathbf{v}|, \re})$ and $\cV(f^{|\mathbf{v}|,\im})$ intersect non-transversally 
in $\boldsymbol{\phi}$. 
\begin{itemize}
 \item For $n = 2$ this is the case if and only if 
$(\mathbf{w},\boldsymbol{\phi})$ is a multiple root.
\item For $n \geq 3$ this is the case if 
and only if all points of $\cV_\R(I)$ in a neighborhood of $\boldsymbol{\phi}$ are 
critical.
\end{itemize}
Since we assumed $\gamma^{-1}(G(\mathbf{w}))$ to be finite, $\boldsymbol{\phi}$ 
has to be an isolated point in $\cV_\R(I)$, and claim (2) follows. Statement (3) is a 
direct consequence of Theorem \ref{Thm:Main1}, which yields that $\mathbf{w}$ is a 
point in the extended boundary if and only if all points in $\cV(f) \cap \F_{\mathbf{w}} \cong 
\cV_{\R}(I)$ are critical, i.e., if and only if they are all multiple roots for $n = 2$ and
isolated points for $n \geq 3$. Statement (4) follows from (1) -- (3).
\end{proof}

Observe that the ideal $I$ in the previous proof is of dimension zero if and only if $n = 2$, which is extremely relevant for the computation of amoebas, their (extended) boundaries and their contours. More precisely, for $n>2$ the corresponding variety $\cV(I)$ is not finite and one has to decide whether its real locus $\cV_\R(I)$ is finite. This is computationally difficult; see Section \ref{Sec:Computation} we see that for further details. 

Second, Theorem \ref{Thm:BoundaryArbitraryDimension} implies that if the $\Log|\cdot|$ map restricted to the variety of a curve in $\C[z_1,z_2]$ is 2 to 1, then the contour of $\cA(f)$ coincides with the extended boundary of the amoeba. Namely, in this case $\# \cV_\R(I) = 2$ and hence either all roots in a fiber are multiple roots or none of them are. This fact partially re-proves Proposition \ref{Prop:Harnack} by Mikhalkin and Rullg{\aa}rd. 

Third, Theorem \ref{Thm:BoundaryArbitraryDimension} leads to a new proof of Proposition \ref{Prop:LinearCase} by Forsberg, Passare and Tsikh as follows.

\begin{proof}\textit{(Proposition \ref{Prop:LinearCase})}
Let $f = 1 + \sum_{j = 1}^n b_j z_j$. After a parameter transformation we can assume that 
every $b_j$ is real. Since $\vol(\New(f)) = 1$ the logarithmic Gau\ss \ map is 1 to 1 by 
Lemma \ref{Lem:LogGaussFinite} and hence, by Theorem \ref{Thm:Main1} and Corollary 
\ref{Cor:LogGaussFormula}, if $\mathbf{w} \in \partial^{e} \cA(f)$, then $\#(\F_{\mathbf{w}} 
\cap \cV(f)) = 1$.
Thus, for $\Log|\mathbf{v}| = \mathbf{w}$ the fiber function $f^{|\mathbf{v}|}$ vanishes at a single point.

It is easy to see by the linearity of $f$ that the image of the fiber function 
$f^{|\mathbf{v}|}$ is a closed annulus around the constant term $1$, where the inner circle 
may have radius zero. Thus, the only uniquely attained values are the two intersection 
points of the outer boundary circle of the annulus with the real line. These extremal points are 
obviously attained if and only if the condition $|b_k v_k| = 1 + \sum_{j \in \{1,\ldots,n\} \setminus 
\{k\} } |b_j v_j|$ is satisfied for some $k$ at a point $\mathbf{v} \in (\C^*)^n$.
\end{proof}

In the remainder of this section we discuss further consequences of Theorem 
\ref{Thm:BoundaryArbitraryDimension}. First, we have a closer look at the topology 
of the variety $\cV(f)$ restricted to a certain fiber $\F_{\mathbf{\Log|\mathbf{v}|}}$, 
which is isomorphic to  $\cV(f^{|\mathbf{v}|})$. By Lemma \ref{Lem:DimFiberManifolds} we 
know that the realification of $\cV(f^{|\mathbf{v}|})$ is a real ($n-2$)-dimensional 
algebraic set. But, as we mentioned before, this set is not connected in general. Thus, 
for every $f \in \C[\mathbf{z}]$ there exists a canonical, non-trivial map 
\begin{eqnarray}
	\struc{B}: \R^n \to \Z, \quad \mathbf{w} \mapsto b_0(\cV(f) \cap \F_{\mathbf{w}}),\label{Equ:BettiMap}
\end{eqnarray}
where $\struc{b_0(\cV(f) \cap \F_{\mathbf{w}})}$ denotes the \struc{$0$-\textit{th Betti number}} of $\cV(f) \cap
\F_{\mathbf{w}}$. This means $B$ maps every point $\mathbf{w}$ of the amoeba ambient space to
the number of connected components of the variety $\cV(f)$ restricted to the fiber of
$\mathbf{w}$ with respect to the $\Log|\cdot|$ map. Thus, the
zero set of $B$ is exactly the complement of the amoeba of $f$. Note that for $n = 2$ we have
$b_0(\cV(f) \cap \F_{\mathbf{w}}) = \# (\F_{\mathbf{w}} \cap \cV(f))$ counted without multiplicity, since by Theorem
\ref{Thm:BoundaryArbitraryDimension} $\F_{\mathbf{w}} \cap \cV(f)$ is isomorphic to a finite arrangement of
points in $(S^1)^2$ in this case.

Obviously, $B$ induces a partition of $\R^n$, particularly of $\cA(f)$, which we
call \struc{$b_0$-\textit{decomposition}} or \struc{\textit{Betti decomposition}}. The following theorem shows that the Betti decomposition yields an alternative way to describe the contour of an amoeba.

\begin{thm}
Let $f \in \C[\mathbf{z}]$ and $\mathbf{w} \in \R^n$. If $\mathbf{w}$ is not an isolated singular point of $\cC(f)$, then $\mathbf{w} \in \cC(f)$ if and only if for all $\eps > 0$ the map $B$ is not constant in the $\eps$-neighborhood $\cB_{\eps}(\mathbf{w})$ of $\mathbf{w}$.
\label{Thm:BettiDecompositionContour}
\end{thm}

Geometrically, the latter theorem means that the contour of a polynomial $f$ is the union of the boundary of (the closure of) the cells given by the Betti decomposition.

\begin{proof}
Assume that there exists no $\eps > 0$ such that $B$ is constant in the $\eps$-neighborhood $\cB_{\eps}(\mathbf{w})$ of $\mathbf{w}$. By
definition, this means that the number of connected
components of the real locus of $\cV(f^{\re}) \cap \cV(f^{\im}) \cap \F_{\mathbf{w}'}$ is not constant for all $\mathbf{w}' \in
\cB_{\eps}(\mathbf{w})$. This means that there exists a point
$\mathbf{w}' \in \cB_{\eps}(\mathbf{w})$ such that for $\Log|\mathbf{v}'| = \mathbf{w}'$ the varieties $\cV(f^{|\mathbf{v}'|, \re})$ and
$\cV(f^{|\mathbf{v}'|, \im})$ do not intersect transversally on some connected component on $\F_{\mathbf{w}'}$, which means that $\mathbf{w}' \in
\cC(f)$ by Theorem \ref{Thm:BoundaryArbitraryDimension}. Since this is true for every $\eps > 0$, it follows that $\mathbf{w}' = \mathbf{w}$. 
 
Assume conversely that there exists an $\eps$-neighborhood of $\mathbf{w}$ where $B$ is constant and $\mathbf{w}$ is a contour point, which is not an isolated singularity. First, assume $\mathbf{w}$ is a regular point of the contour of $\cA(f)$. It follows from Proposition \ref{Prop:SameTangentSpace} and its proof that a connected component of the real locus of $\cV(f^{\re}) \cap \cV(f^{\im}) \cap \F_{\mathbf{w}}$, which is given by some critical point of $f$ with respect to the $\Log|\cdot|$ map, vanishes if we move from $\mathbf{w}$ in the corresponding normal direction. This is a contradiction to the assumption that $B$ is constant. 

If $\mathbf{w}$ is a singular point of the contour but not isolated, then every $\eps$-neighborhood contains regular points of the contour. Hence, the same argument as above finishes the proof.
\end{proof}

\begin{cor}
Let $f \in \C[\mathbf{z}]$. Then the following statements are equivalent
\begin{enumerate}
 \item The smooth loci of $\cC(f)$ and $\partial^{e} \cA(f)$ coincide.
 \item $B$ is constant on $\cA(f) \setminus \partial^{e} \cA(f)$.
 \item The Betti decomposition induced by $B$ has exactly one connected component on $\cA(f) \setminus \partial^{e} \cA(f)$.
\end{enumerate}
\label{Cor:BettiPartition}
\end{cor}

\begin{flushright}
	$\square$
\end{flushright}

\medskip

We show now that we can bound the map $B$ from above for $n=2,3$.
The \struc{\textit{Harnack inequality}} \cite{Harnack}, \cite[Cor. 5.4]{Huisman} implies that a real
algebraic curve of genus $g$ can have at most $g+1$ real connected components.
By a result due to Castelnuovo, e.g., \cite{Miranda,Pecker}, the genus of a real algebraic
curve of degree $d$ in $\P^n$ is bounded from above by \struc{\textit{Castelnuovo's bound}
$C(d,n)$}. $C(d,n)$ is defined as follows: for $d\geq n\geq 2$ there exist unique $m,p\in
\N$ such that $d-1=m(n-1)+p$ with $0\leq p \leq n-1$. Then $C(d,n)=m(\frac{n-1}{m-1}+2
p)/2$. For $n=2$ this bound specializes to the well-known result $C(d,2)=\frac{(d-1)(d-2)}{2}$, which can
also be obtained by the adjunction formula for $g$.

\begin{thm}
Let $f \in \C[\mathbf{z}]$ with $\deg(f) = d$, $\deg(f^{\re}) = d^{\re}$ and
$\deg(f^{\im}) = d^{\im}$. Let $B$ denote the map yielding the associated $b_0$-partition
of $\R^n$ as in \eqref{Equ:BettiMap}. Let $\mathbf{w} \in \R^n$ such that $\cV(f)$ and $\F_{\mathbf{w}}$ do not have a common irreducible component.
\begin{enumerate}
	\item If $n = 2$ then $B(\mathbf{w}) \leq 4 d^{\re} d^{\im} \leq 4 d^2$ for all $\mathbf{w} \in \R^2$.
	\item If $n = 3$ then $B(\mathbf{w}) \leq C(4d^2,6)+1$ for all $\mathbf{w} \in
\R^3$.
\end{enumerate}
\label{Thm:BettiUpperBound}
\end{thm}

\begin{proof}
Intersection results like B\'{e}zout's Theorem for projective varieties also hold in our affine situation, since we work with \textit{compact} affine subvarieties of the fiber torus $\F_{\mathbf{w}}$.

By the definition of $B$ and Theorem \ref{Thm:BoundaryArbitraryDimension}, $B(\mathbf{w})$ equals
the number of connected components of the real locus of the variety of the ideal $I = \langle f^{\re},f^{\im},
x_1^2 + y_1^2 - |v_1|^2, \ldots, x_n^2 + y_n^2 - |v_n|^2 \rangle$, where $\Log|\mathbf{v}| = \mathbf{w}$. Since $\cV(f)$ and $\F_{\mathbf{w}}$ do not have a common irreducible component by assumption, for $n=2$ the intersection $\cV(f) \cap
\F_{\mathbf{w}}$ is zero-dimensional. By B\'{e}zout's Theorem, $B(\mathbf{w})$ has degree $D\leq 4 d^{\re} d^{\im} \leq
4 d^2$. In the case $n=3$ the variety $\cV(f) \cap \F_{\mathbf{w}}$ is a curve, hence we can apply Castelnuovo's bound to the real locus of the variety $\cV(I)$.
\end{proof}

\section{Computation of the Boundary}
\label{Sec:Computation}

The computation of amoebas, its contour and its (topological) boundary was initialized by Theobald in \cite{Theobald:ComputingAmoebas} and tackled by different authors since then. In
this section, we first give an overview about the known methods for the computation of amoebas. Afterwards we use our own results, particularly Theorem \ref{Thm:BoundaryArbitraryDimension} to provide a new algorithmic approach. This approach allows us to
decide membership of points in amoebas \textit{exactly} and yields a distinction between
extended boundary points and contour points. While the corresponding computation in dimension at least three involves quantifier elimination methods, we only need Gr\"obner basis methods in dimension two and hence an
efficient computation of the extended boundary of the amoeba, and even the Betti
decomposition of amoebas is possible. Note that no other known
algorithm can can compute the boundary or decide membership (particularly not with symbolic methods). We also present results of a prototype implementation of our algorithm. \\

We start with a comparison of the existing approximation methods for amoebas. As mentioned
before, the first method was given by Theobald \cite{Theobald:ComputingAmoebas} in
dimension two and can be generalized to higher dimensions. The key idea is to use
Mikhalkin's Theorem \ref{Thm:MikhalkinContour} to compute the contour of
an amoeba and thus obtain an approximation. Practically, for $f \in \C[\mathbf{z}]$ this
can be done by computing the roots of all ideals
\begin{eqnarray}
	I_{\mathbf{s}} & = & \left\langle f, z_n \cdot \frac{\partial f}{\partial z_n} -
s_1 \cdot z_1 \cdot \frac{\partial f}{\partial z_1},\ldots,  z_n \cdot \frac{\partial
f}{\partial z_n} -
s_{n-1} \cdot z_{n-1} \cdot \frac{\partial f}{\partial z_{n-1}}\right\rangle
\label{Equ:ThorstensIdeal}
\end{eqnarray}
with $\mathbf{s}=(s_1, \ldots,s_{n-1}) \in \R^{n-1}$. Theorem \ref{Thm:MikhalkinContour}
guarantees that all points in $\cV(I)$ are critical and all critical points are of the
Form
\eqref{Equ:ThorstensIdeal}. Since finally $\cV(I)$ is zero-dimensional one can obtain an
approximation of the contour and hence of the amoeba by computing $\cV(I)$ for suitable
many different $\mathbf{s} \in \R^{n-1}$, see Lemma \ref{Lem:LogGaussFinite}.

Later approaches rely on solving the following \struc{\textit{membership problem}}.
\begin{prob}
Let $f \in \C[\mathbf{z}]$ and $\mathbf{w} \in \R^n$. Decide, whether $\mathbf{w} \in \cA(f)$.
\label{Prob:MembershipProblem}
\end{prob}

Although already described in \cite{Theobald:ComputingAmoebas}, this approach was first
used by Purbhoo \cite{Purbhoo} to provide an approximation method for amoebas based on a
\struc{\textit{lopsidedness certificate}}. This certificate checks whether, evaluated at one
specific point, the absolute value of one monomial of $f$ is larger than the sum of the
absolute values of all other monomials. It is easy to see that if this is the case, then
the point is contained in the complement of the amoeba. In general the contrary is not true: a point on the complement of the amoeba is not necessarily lopsided (except for linear polynomials, see Proposition \ref{Prop:LinearCase}). Purbhoo, however, showed that it is possible to investigate iterated resultants of the original polynomial in order to approximate the amoeba. The iteration process keeps
the amoeba invariant, but guarantees that every point becomes lopsided if the iteration level tends towards infinity \cite{Purbhoo}.

Theobald and the second author showed in \cite{Theobald:deWolff:SOS} that the membership problem can be transformed into a feasibility problem of a particular \struc{\textit{semidefinite optimization problem (SDP)}} via realification and using the Real Nullstellensatz. This allows an approximation of amoebas via SDP-methods, which turn out to be of at most the same complexity as Purbhoo's approach.

Avenda\~{n}o, Kogan, Nisse and Rojas \cite{Nisse:Rojas:et:al} provided limit bounds on the (Hausdorff-) distance between the amoeba of a polynomial $f = \sum_{j = 1}^d b_j \mathbf{z}^{\alp(j)}$ and
the tropical hypersurface of the tropical polynomial $\struc{\archtrop(f)} = \bigoplus_{j = 1}^d \log|b_j| \odot \mathbf{z}^{\alp(j)}$. Recall that the tropical polynomial $\archtrop(f)$ is defined
over the \struc{\textit{tropical semi-ring}} $\struc{(\R \cup \{-\infty\},\oplus,\odot)} = (\R \cup \{-\infty\},\max,+)$, i.e.,  $\archtrop(f) = \max_{j = 1}^d \log|b_j| + \langle \mathbf{z}, \alp(j) \rangle$,
where $\struc{\langle \cdot, \cdot \rangle}$ denotes the standard inner product. Recall furthermore that the \struc{\textit{tropical variety}} of a tropical polynomial is defined as the subset of $\R^n$
where the tropical polynomial has a non-smooth image, i.e., as subset of $\R^n$, where the
maximum is attained at least twice. Since the tropical hypersurface given by $\archtrop(f)$ is easy
to compute, this result yields a new rough but quick way to approximate the amoeba via
finding certificates for non-containment in the amoeba for
certain points. For an introduction to tropical geometry see \cite{Maclagan:Sturmfels}.\\

Now, we discuss our approach. The main improvement is that we
are able to describe the extended boundary of amoebas and that we are able to answer the membership
question \textit{exactly}. The former approaches only yield certificates for non-membership
of points in the amoeba, but cannot certify membership. Our approach is to use
Theorem \ref{Thm:BoundaryArbitraryDimension}, which guarantees for $n = 2$ that every
fiber $\F_{\mathbf{w}}$ only contains finitely many points of $\cV(f)$. Furthermore,
it allows us to distinguish between regular points of the amoeba, its contour points and its extended boundary
points. 

\begin{cor}
Let $f \in \C[\mathbf{z}]$ such that the real and imaginary parts of the coefficients
of $f$ are rational, and $\mathbf{w}\in \R^n$ such that $\exp(\mathbf{w}) \in \Q^n$.
Then we can decide with symbolical methods, whether $\mathbf{w} \in \cA(f)$, $\mathbf{w}
\in \partial^{e} \cA(f)$ or $\mathbf{w} \in \cC(f)$.
\label{Cor:FiberComputation}
\end{cor}

Note that this corollary means in particular that we can solve the membership problem
exactly.

\begin{proof}
Let $\mathbf{v}=(v_1,\ldots,v_n)\in \Q^n$ such that $\Log|\mathbf{v}| = \mathbf{w}
\in \R^n$. We also set $z_j = x_j + \textnormal{i}\cdot y_j$ for $j\in\{1,\ldots,n\}$.\\
Suppose $n=2$. We compute the zero-dimensional variety of $I = \langle
f^{\re}, f^{\im}, x_1^2 + y_1^2 - |v_1|^2, x_2^2 + y_2^2 - |v_2|^2 \rangle$ symbolically
via computing a Gr\"{o}bner basis. Using Theorem \ref{Thm:BoundaryArbitraryDimension} we
can read out from the real locus where $\mathbf{w}$ is located.\\
If $n>2$, then we also consider the ideal $I = \langle f^{\re}, f^{\im}, x_1^2 + y_1^2 -
|v_1|^2,
\ldots, x_n^2 + y_n^2 - |v_n|^2 \rangle$. The generic dimension of $I$ is $n-2>0$, but by
Theorem \ref{Thm:BoundaryArbitraryDimension} it suffices to determine the dimension of
the real locus of $\cV(I)$. This can be decided by quantifier elimination methods; see \cite[Algorithm 14.10.]{Basu:Pollack:Roy}.
\end{proof}

With Corollary \ref{Cor:FiberComputation} we can approximate the contour \textit{and} the extended boundary of an amoeba $\cA(f)$ of a given polynomial $f$ in dimension
two. First, we compute the contour $\cC(f)$, for example by using Theobald's method solving the ideal $I$ from \eqref{Equ:ThorstensIdeal}.
Alternatively, we can compute all contour points along a one dimensional affine subspace of $\R^2$ given by fixing the absolute value of $z_1$ or $z_2$. This means we compute the variety of the ideal
\begin{eqnarray*}
	\left\langle f^{\re}, f^{\im}, \left(z_1 \cdot \frac{\partial f}{\partial z_1}\right)^{\re} \cdot \left(z_2 \cdot \frac{\partial f}{\partial z_2}\right)^{\im} - \left(z_1
\cdot \frac{\partial f}{\partial z_1}\right)^{\im} \cdot \left(z_2 \cdot \frac{\partial f}{\partial z_2}\right)^{\re}, x_1^2 + y_1^2 - \lam_1^2 \right\rangle
\end{eqnarray*}
in $\R[x_1,x_2,y_1,y_2]$, where the third polynomial guarantees that we only investigate
critical points and the fourth polynomial ensures that the absolute value of $z_1$
equals $\lam_1 \in \R$. In the following example we present the result of a prototype \textsc{Sage} / \textsc{Singular}
\cite{sage,Singular} implementation of our algorithm.\\

For $n>2$ an implementation to approximate the boundary is also possible using
Theobald's method and Corollary \ref{Cor:FiberComputation} or by computing the Betti decomposition. However, the runtime of such an algorithm is very long due to the necessary quantifier elimination.

\begin{exa}
We approximate the amoebas of $f = z_1^2z_2 + z_1 z_2^2 + c z_1 z_2 + 1$ with $c = 1.5$ and $c = 1$ via computing their contour and their boundary. Here, we use Theobald's method , see \eqref{Equ:ThorstensIdeal}, to compute the contour $\cC(f)$. Afterwards, we distinguish between contour and boundary with the approach described in Corollary \ref{Cor:FiberComputation}. The result is shown in Figure \ref{Fig:AmoebaBoundaryApproximation}.
\endenvi
\end{exa}

\begin{figure}
\ifpictures
\includegraphics[width=0.4\linewidth]{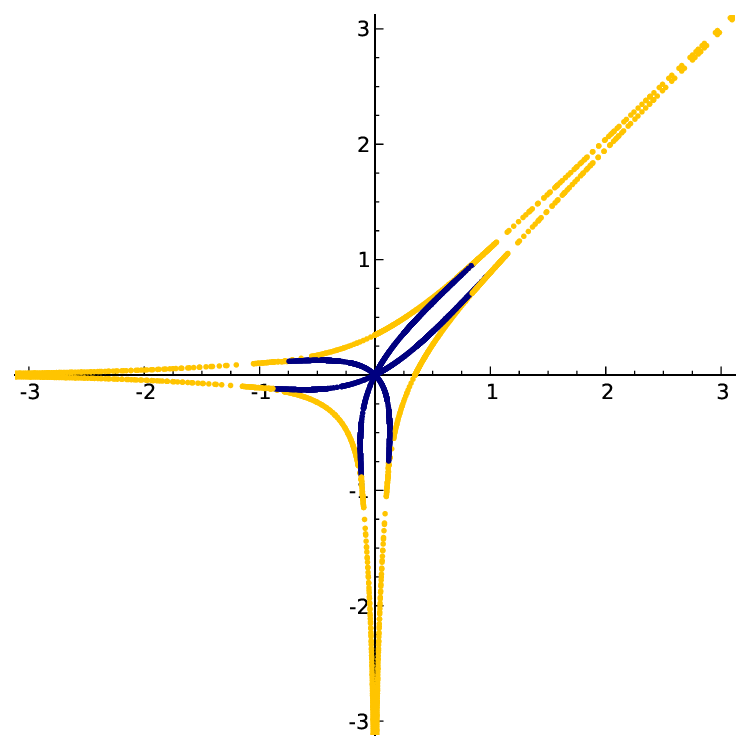} \quad
\includegraphics[width=0.4\linewidth]{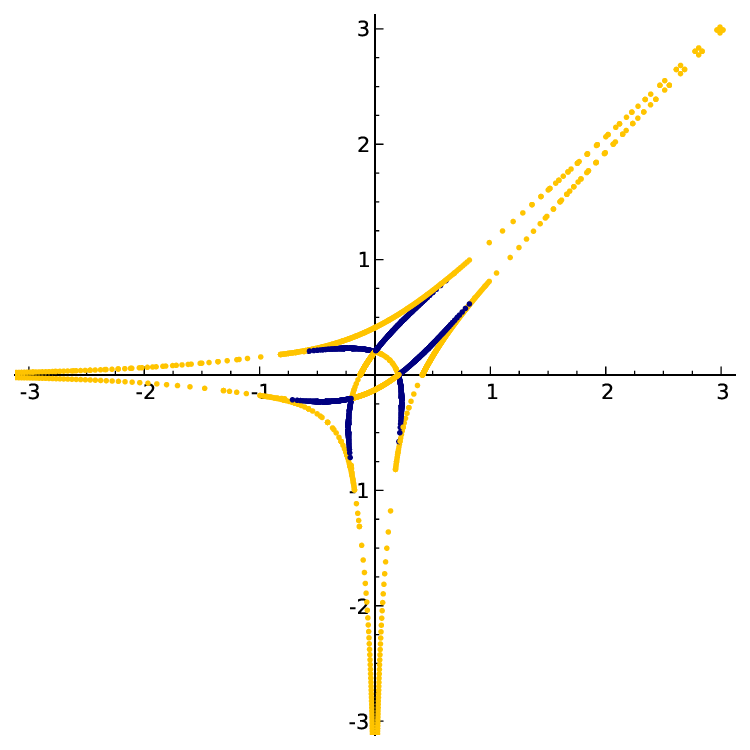}
\fi
\caption{A prototype implementation of our method: An approximation of the contour (dark) and the (extended) boundary (light) of the amoebas of the polynomials $f = z_1^3 + z_2^3 + c z_1 z_2 + 1$ with $c = 1$ and $c = 1.5$.}
\label{Fig:AmoebaBoundaryApproximation}
\end{figure}

Corollary \ref{Cor:FiberComputation} allows us furthermore to compute an arbitrary exact approximation of the Betti decomposition of the amoeba of an arbitrary polynomial in dimension two. Recall that in dimension two the map $B$, see \eqref{Equ:BettiMap}, maps $\mathbf{w} \in \R^n$ to the number of roots in $\cV(f) \cap \F_{\mathbf{w}}$, i.e., to the number of real zeros in a fiber ideal as given in Theorem \ref{Thm:BoundaryArbitraryDimension}. Hence, we can compute all values of $B$ via Corollary \ref{Cor:BettiPartition}, which yields, in addition to information provided by the image of $B$ itself, another method to approximate the amoeba, its contour, and its boundary. The following example is computed via a prototype \textsc{Sage} / \textsc{Singular} implementation of our method.

\begin{exa}
We compute the Betti decomposition of the polynomials $f = z_1^2z_2 + z_1 z_2^2 + c z_1 z_2 + 1$ with $c = -4$ and $c = 1.5$. In order to do so, we compute the variety of every fiber
ideal of points contained in the set $\{20 \cdot (w_1,w_2) \in \Z^2 \ : \ -2 \leq w_1,w_2 \leq 2 \}$. The plots are depicted in Figure \ref{Fig:BettiDecompositon}. Note that for $c = -4$
the polynomial $f$ defines a Harnack curve and hence $\Log|\cdot|$ is 2 to 1 by Proposition \ref{Prop:Harnack} or by Theorem \ref{Thm:BoundaryArbitraryDimension}. Thus, by Corollary
\ref{Cor:BettiPartition} we have $\partial^{e} \cA(f) = \cC(f)$ in this case. Therefore, the left picture in
Figure \ref{Fig:BettiDecompositon} also provides an example for this Corollary \ref{Cor:BettiPartition}.
\label{Exa:BettiDecomp}
\endenvi
\end{exa}

\begin{figure}
\ifpictures
\includegraphics[width=0.4\linewidth]{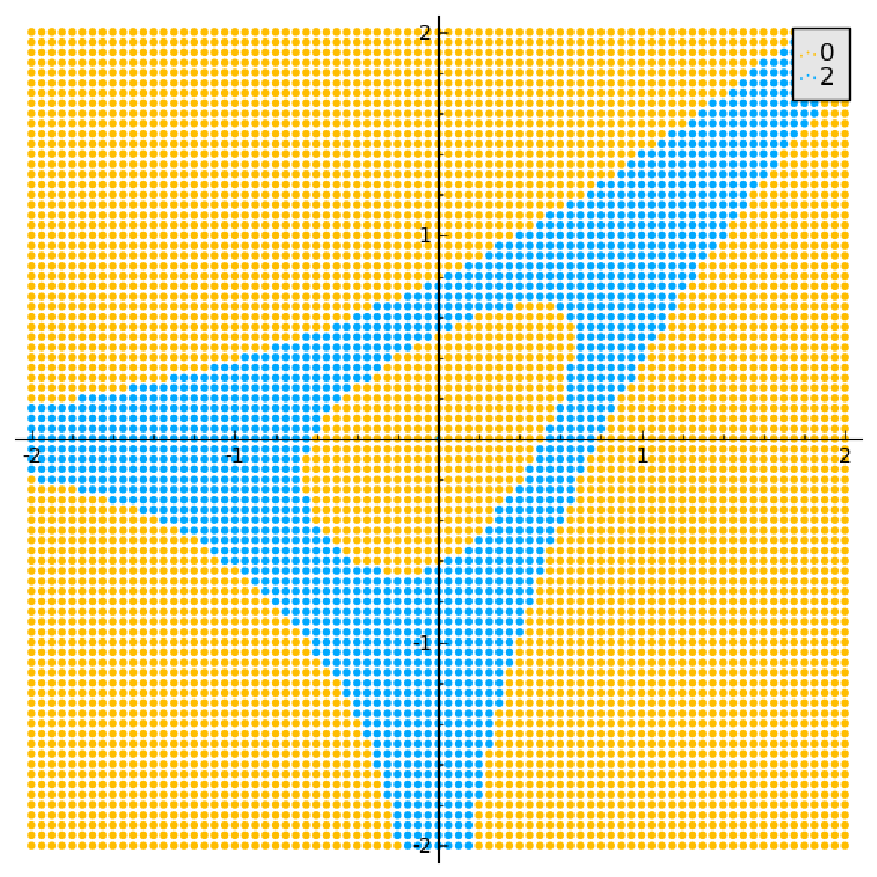} \quad
\includegraphics[width=0.4\linewidth]{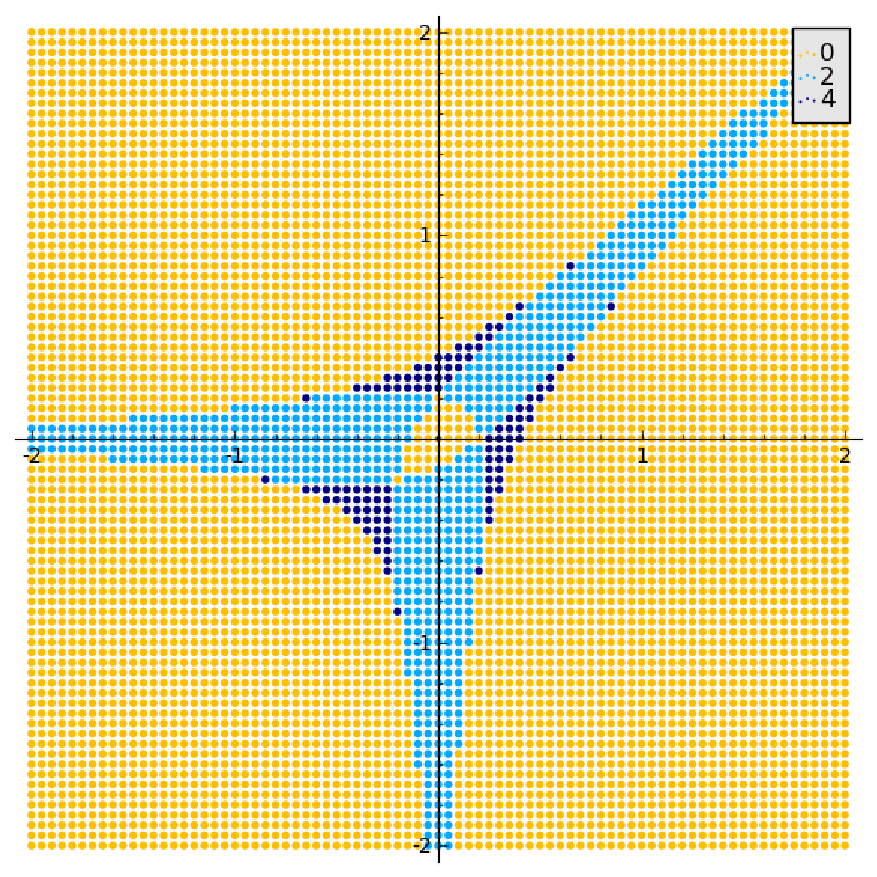}
\fi
\caption{A prototype implementation to approximate of the Betti decomposition of $f = z_1^3 + z_2^3 + c z_1 z_2 + 1$ for $c = -4$ and $c = 1.5$.}
\label{Fig:BettiDecompositon}
\end{figure}

\section{Impact on Amoeba Bases}
\label{Sec:AmoebaBases}

In contrast to amoebas of Laurent polynomials and their hypersurfaces, which have been extensively studied and are well understood in many aspects, amoebas of arbitrary ideals, more precisely, the $\Log|\cdot|$ image of varieties of codimension larger than one, have merely been investigated and are almost completely not understood. To the best of our knowledge, the most remarkable result is by
Purbhoo, see \cite[Corollary 5.2]{Purbhoo}, stating that for every polynomial ideal $I$ it holds
\begin{eqnarray*}
	\cA(I) & = & \bigcap_{f \in I} \cA(f).
\end{eqnarray*}

A canonical question arising from this result, which was already mentioned by Purbhoo himself, is if the intersection on the right hand side can be restricted to a \textit{finite} subset of the polynomials in the ideal, an \struc{\textit{amoeba basis}}. 

Before we define amoeba bases, we briefly recall the definition of Gr\"{o}bner bases and tropical bases in order to demonstrate the analogy in the
definition of amoeba bases in \ref{Def:AmoebaBasis} afterwards. In algebraic geometry and computer algebra Gr\"{o}bner bases play a fundamental
role, since they are the standard tool to solve (non-linear) systems of polynomial equations. For a finitely generated ideal $I = \langle
f_1,\ldots,f_r \rangle \subseteq \C[\mathbf{z}]$ with variety $\cV(I)$ a \struc{\textit{Gr\"{o}bner basis}} $\struc{G} \subset I$ (with respect to a given monomial
ordering $\prec$) is a finite system of polynomials $g_1,\ldots,g_s$ such that the ideal $\lt_\prec(I)$ of leading terms of polynomials in
$I$ is generated by the leading terms of the elements of the Gr\"{o}bner basis. This means
\begin{eqnarray*}
	\lt_\prec(I) & = & \langle \lt_\prec(g_1),\ldots,\lt_\prec(g_s)\rangle.
\end{eqnarray*}
Gr\"obner bases exist for every finitely generated polynomial ideal and monomial ordering and can be computed efficiently, for example via Buchberger's algorithm. For an introduction to the topic see \cite{Cox:Little:OShea}.

Similarly, in tropical geometry there exist \textit{tropical bases}. Let $I \subset \K[\mathbf{z}]$ be an ideal generated by finitely many
polynomials in a polynomial ring $\K[\mathbf{z}]$ over a real valuated field $\K$ like the field of Puiseux series. One defines for every $f =
\sum_{\alp \in A} b_\alp \mathbf{z}^{\alp} \in \K[\mathbf{z}]$ with $A \subset \Z^n$ the corresponding tropical polynomial $\trop(f)$ as  
\begin{eqnarray*}
	\struc{\trop(f)} & = & \bigoplus_{\alp \in A} -\val(b_\alp) \oplus \mathbf{z}^{\alp},
\end{eqnarray*}
where $\val$ denotes the natural valuation map from the algebraic closure $\ovl{\K}$ of $\K$ to $\R \cup \{\infty\}$. As before, the tropical variety $\cT(\trop(f))$ is defined as the subset of $\R^n$ where $\trop(f)$ attains its maximum at least twice and the tropical variety $\cT(I)$ of the ideal $I$ is given by $\cT(I) = \bigcap_{f \in I} \cT(\trop(f))$. With this notation one calls $G = (g_1,\ldots,g_r) \subset \K[\mathbf{z}]$ a \struc{\textit{tropical basis}} of $I$ if $\langle g_1,\ldots,g_r \rangle = I$ and
\begin{eqnarray*}
	\cT(I) & = & \bigcap_{j = 1}^r \cT(\trop(g_j)).
\end{eqnarray*}
For more details about tropical bases see \cite{Bogart:et:al,Hept:Theobald:2,Maclagan:Sturmfels,Speyer:Sturmfels}.\\

As an analog to Gr\"{o}bner bases from algebraic geometry and tropical bases from tropical geometry, we define an amoeba basis in the following way. 

\begin{defn}
Let $I \subset \C[\mathbf{z}]$ be a finitely generated ideal. Then we call $(g_1,\ldots,g_s) \subset I$ an \struc{\textit{amoeba basis}} if
\begin{eqnarray*}
	\begin{array}{crcl}
		(1) & \cA(I) & = & \bigcap_{j = 1}^s \cA(g_j), \\
		(2) & \cA(I) & \subset & \bigcap_{j \in \{1,\ldots,s\} \setminus \{i\}} \cA(g_j) \text{ for every } 1 \leq i \leq s, \\
		(3) & \langle g_1,\ldots,g_s \rangle & = & I.\\
	\end{array}
\end{eqnarray*}
\label{Def:AmoebaBasis}
\endenvi
\end{defn}

Although the question of the existence and (possible) form of amoeba bases was roughly stated by Purbhoo \cite{Purbhoo} (without using the term), to the best of our knowledge no formal definition of an amoeba basis was given elsewhere before. Therefore, we point out, while axiom (2) only requires minimality of an amoeba basis and hence is no proper restriction, it is not clear whether it makes sense to require axiom (3) in general. We do this here since we want to create an object, which is defined analogously to Gr\"{o}bner
bases and tropical bases.

It is unclear for which ideals amoeba bases exist. Purbhoo claims that they do not exist in general \cite[p. 25]{Purbhoo}, but he does not give a formal proof. Furthermore, if amoeba bases exist for certain ideals, then it is unclear how many elements they have, how they can be computed, or if they are unique in any sense. In this section we show that if an amoeba basis exists for some given ideal, then the amoeba of the ideal intersects the boundary of the amoeba of every basis element, see Theorem \ref{Thm:AmoebaBasesBoundary}. Thus, an understanding of the boundary of amoebas is \textit{crucial} in order to find amoeba bases. Furthermore, we demonstrate that amoeba bases are a useful concept in general. Namely, in Theorem \ref{Thm:AmoebaBasesLinearCase} we give an algorithm that computes a linear amoeba basis of length $n+1$ for every ideal given by a full ranked linear system of equations in $(\C^*)^n$.\\

As a first result of this section, we show that the boundary of an amoeba plays a key role for the comprehension of amoeba bases.

\begin{thm}
Let $I \subset \C[\mathbf{z}]$ be a finitely generated ideal. Assume there exists an amoeba basis $(g_1,\ldots,g_s)$ for $I$. Then
$\partial^{e} \cA(g_i) \cap \cA(I) \neq \emptyset$ for every $1 \leq i \leq s$.
\label{Thm:AmoebaBasesBoundary}
\end{thm}

\begin{proof}
Since $(g_1,\ldots,g_s)$ is an amoeba basis of $I$, we have $\cA(I) = \bigcap_{j = 1}^s \cA(g_j)$ and therefore $\cA(I) \subset \cA(g_i)$ for every $i$. Assume $\partial^{e} \cA(g_i) \cap \cA(I) = \emptyset$ for some $i$. Then $\cA(I) \subset (\cA(g_i) \setminus \partial \cA(g_i))$ and thus $\cA(I) = \bigcap_{j \in \{1,\ldots,s\} \setminus \{i\}} \cA(g_j)$. This is a contradiction
to the minimality assumption of amoeba bases.
\end{proof}

In Theorem \ref{Thm:AmoebaBasesBoundary} we assumed that amoeba bases exist. But, as already mentioned in the beginning of the section, the question about the existence of amoeba bases is open. Obviously, amoeba bases always exist for principal ideals and hence particularly for ideals of univariate polynomials. But this is trivial since  we have $\cA(I) = \cA(f)$ for every such ideal $I = \langle f \rangle$. We show in the following that in general there also exist amoeba bases for non-trivial cases by proving their existence and
computability for full ranked systems of linear equations. The following theorem is joint work with Chris Manon. The initial proof strategy was obtained by him and the second author.

\begin{thm}
Let $I = \langle f_1,\ldots,f_{n} \rangle \subset \C[\mathbf{z}]$, where the $f_j$ are linear polynomials defining a system of linear equations of full rank satisfying $\cV(I) = \mathbf{v}$. Then 
\begin{eqnarray*}
	g_0 & = & 1 + \frac{1}{||\mathbf{v}||_1} \sum_{k = 1}^n -e^{-\textnormal{i} \cdot \arg(v_k)} z_k, \\
	g_j & = & 1 + \sum_{k \in \{1,\ldots,n\} \setminus \{j\}}^n e^{-\textnormal{i} \cdot \arg(v_k)} z_k - \frac{1 + ||\mathbf{v}||_1 -
|v_j|}{v_j} \cdot z_j \ \text{ for } 1 \leq j \leq n, 
\end{eqnarray*}
with $||\mathbf{v}||_1 = \sum_{k = 1}^n |v_k|$ is a linear amoeba basis of length $n+1$ for $\cA(I)$.
\label{Thm:AmoebaBasesLinearCase}
\end{thm}

\begin{proof}
 Let $\Log|\mathbf{v}| = \mathbf{w} \in \R^n$ and hence $\cA(I) = \{\mathbf{w}\}$. First, notice that every $g_j(\mathbf{v}) = 0$ by construction and since all $g_j$ are linear this implies $\cV(\langle g_0,\ldots,g_n \rangle) = \mathbf{v}$. Since furthermore all $f_j$ and $g_j$ are linear, $I$ and $\langle g_0,\ldots,g_n \rangle$ equal their radical ideals and thus $\langle g_0,\ldots,g_n \rangle = I$.

Furthermore, due to the choice of coefficients, we have for every $g_j(\mathbf{v})$ that the norm of the term in $z_j$ for $j \geq 1$ and the constant term for $j = 0$ equals the sum of the norms of all other terms of $g_j(\mathbf{v})$. More specifically, we have
\begin{eqnarray*}
	\frac{1}{||\mathbf{v}||_1} \cdot \sum_{k = 1}^n |- e^{\textnormal{i} \cdot \arg(v_k)} \cdot v_k | \ = \ \frac{1}{||\mathbf{v}||_1}
\cdot \sum_{k = 1}^n |v_k| \ = \ 1 \quad \text{ for } g_0
\end{eqnarray*}
and
\begin{eqnarray*}
	1 + \sum_{k \in \{1,\ldots,n\} \setminus \{j\}} |- e^{\textnormal{i} \cdot \arg(v_k)} \cdot v_k| \ = \ 1 + ||\mathbf{v}||_1 - |v_j| 
\ = \ \left| \frac{1 + ||\mathbf{v}||_1 - |v_j|}{v_j} \cdot v_j \right| \quad \text{ for } g_j.
\end{eqnarray*}
By Proposition \ref{Prop:LinearCase} this implies $\mathbf{w} \in \ovl{E_{e_j}(g_j)} \cap \partial^{e} \cA(g_j)$ for every $0 \leq j \leq n$. Recall that $E_{e_j}(g_j)$ denotes the component of the complement of the amoeba of $g_j$ with order given by the $j$-th standard vector $e_j$, where $e_0 = \mathbf{0}$. Thus, we have $\mathbf{w} \in \bigcap_{k = 0}^n \cA(g_k)$ and hence $(g_0,\ldots,g_n)$
indeed forms an amoeba basis, if we can show that $\bigcap_{k = 0}^n \cA(g_k)$ contains no points besides $\mathbf{w}$. Assume, there exists another $\mathbf{u} \in (\C^*)^n$ with $\Log|\mathbf{u}| \neq \mathbf{w}$ and $\Log|\mathbf{u}| \in \bigcap_{k = 0}^n \cA(g_k) \subset \R^n$. Then either $||\mathbf{u}||_1 < ||\mathbf{v}||_1$ or there exists an entry $u_j$ with $|u_j| > |v_j|$. But this means either for $g_0(\mathbf{u})$ that $1 > \sum_{k = 1}^n |u_k|$ or
for some $g_j(\mathbf{u})$ that $|u_j| > 1 + \sum_{k \in \{1,\ldots,n\} \setminus \{j\}} |u_k|$. Again, by Proposition \ref{Prop:LinearCase}, this implies that there exists some $j \in \{0,\ldots,n\}$ with $\Log|\mathbf{u}| \in E_{e_j}(g_j)$, which is a contradiction to the assumption that $\Log|\mathbf{u}| \in \bigcap_{k = 0}^n \cA(g_k)$. Since $\mathbf{v}$ is computable with the Gau\ss \ algorithm, the computability of the amoeba
basis follows.
\end{proof}

\begin{exa}
Let $I = \langle f_1,\ldots,f_n \rangle$ be a zero-dimensional ideal such that all $f_j$ are linear and the coefficient matrix of the defining set (without the constant terms) is \struc{\textit{stochastic}},
i.e., the entries of the coefficient vector of each $f_j$ are positive real numbers summing up to 1. Since the constant term is always 1 we have $\cV(I) = \{-\mathbf{1}\}$. Then 
$$\left(1 + \frac{1}{n} \sum_{k = 1}^n z_k, 1 - \sum_{k = 2}^n z_k + n z_1,\ldots,1 - \sum_{k = 1}^{n-1} z_k + n z_n\right)$$
is an amoeba basis for $I$. See Figure \ref{Fig:AmoebaBasisLinear} for the case $n = 2$.
\endenvi
\end{exa}

\begin{figure}
\ifpictures
\includegraphics[width=0.5 \linewidth]{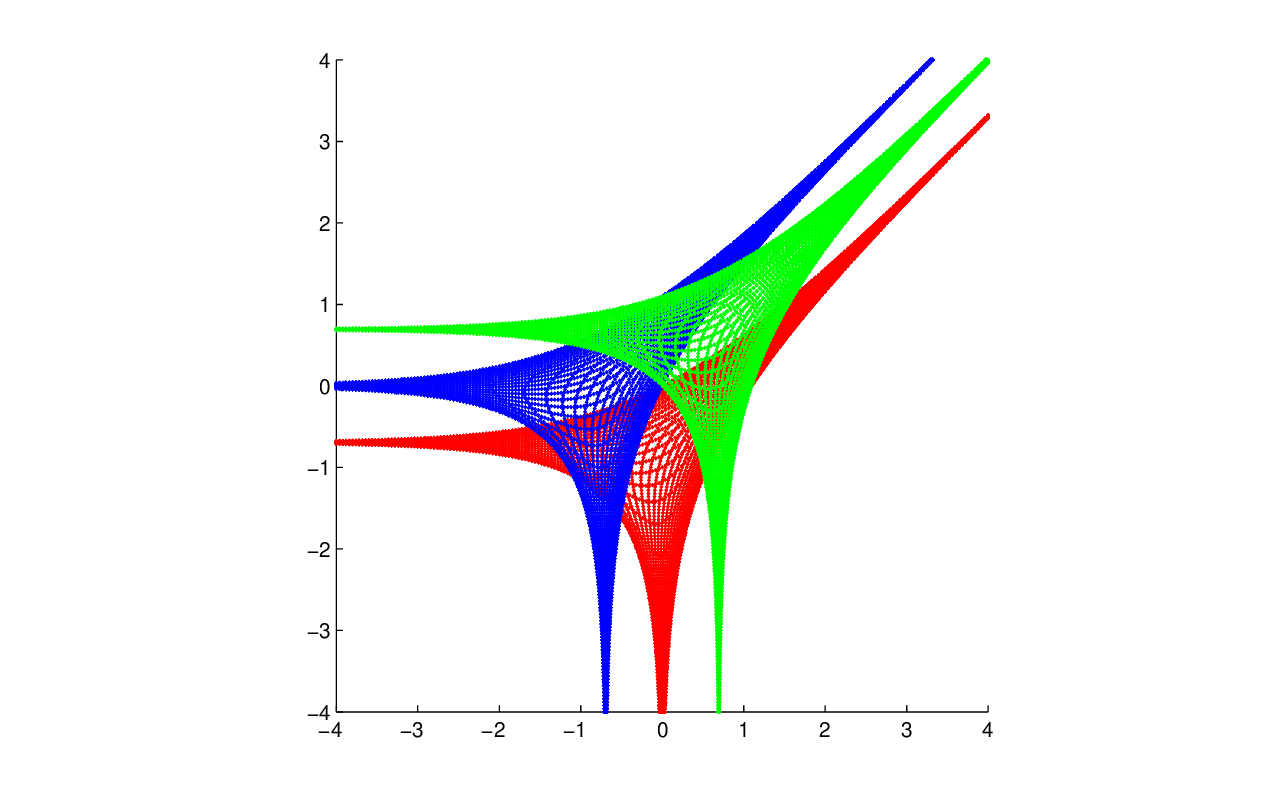}
\fi
\label{Fig:AmoebaBasisLinear}
\caption{The amoeba basis $(1 + 0.5 z_1 + 0.5 z_2, 1 + 2 z_1 - z_2, 1 - z_1 + 2 z_2)$ for the zero-dimensional ideal given by linear polynomials with stochastic $2
\times 2$ coefficient matrix. The figure shows that the amoebas of the basis elements only intersect in the origin, which is exactly $\Log|(-1,-1)|
= \cA(I)$.}
\end{figure}

\medskip

\appendix
\section{Proof of Lemma \ref{Lem:CriticalPointsLogArg}}
\label{Sec:Appendix}

For convenience of the reader we give an own proof of Lemma \ref{Lem:CriticalPointsLogArg}.

\begin{proof}(Lemma \ref{Lem:CriticalPointsLogArg})
We choose a local branch of the holomorphic logarithm $\Log_\C$, and we identify $(\C^*)^n$ with $(\R^2 \setminus \{(0,0)\})^{n}$. We consider $\Log|\cdot|$ and $\Arg$ in an infinitesimal neighborhood $\cB_{\eps}(\mathbf{v})$ around a point $\mathbf{v}$. Thus,  $\Log|\cdot|$ and $\Arg$ behave like linear maps from $\R^{2n}$ to $\R^n$ with
\begin{eqnarray}
	 \cB_{\eps}(\mathbf{v}) & = & \cB_{\eps}(\mathbf{v}) \setminus \Ker(\Log|\cB_{\eps}(\mathbf{v})|) \ \oplus \ \cB_{\eps}(\mathbf{v}) \setminus \Ker(\Arg(\cB_{\eps}(\mathbf{v}))) \label{Equ:Orthogonal} \\
	 & = & \Ker(\Log|\cB_{\eps}(\mathbf{v})|) \ \oplus \ \Ker(\Arg(\cB_{\eps}(\mathbf{v}))). \nonumber
\end{eqnarray}
Let $\mathbf{v} \in \cV(f)$. Consider the tangent space $T_{\mathbf{v}}\cV(f)$ of $\mathbf{v}$ in $\cV(f)$. Since $\cV(f)$ is a variety of complex codimension one, we have after realification $\dim~T_{\mathbf{v}}\cV(f) = 2n-2$.

First, we prove $\dim~\Log|(T_{\mathbf{v}}\cV(f))^\perp| = \dim~\Arg((T_{\mathbf{v}}\cV(f))^\perp)= 1$. Indeed, $(T_{\mathbf{v}}\cV(f))^\perp$
is spanned by the (complex) normal vector $\mathbf{t}$ of $T_{\mathbf{v}}\cV(f)$ and hence we know that $\dim~\Log|(T_{\mathbf{v}}\cV(f))^\perp|$ $+\dim~\Arg((T_{\mathbf{v}}\cV(f))^\perp)= 2$. But we also have that $\dim~\Log|(T_{\mathbf{v}}\cV(f))^\perp| \leq 1$ as $\Log|\mathbf{t}|$ is given by $\RE(\Log_{\C}(\mathbf{t}))$, which is locally a linear map. Thus, the claim follows.

Now, let $\mathbf{v} \in \cV(f)$ be critical under the $\Log|\cdot|$ map. We investigate the situation locally so that $\Log|\cdot|$ can be treated as a linear map. Since $\mathbf{v}$ is critical the Jacobian of $\Log|\cdot|$ at $\mathbf{v}$ does not have full rank, i.e.,
$\dim~\Log|T_{\mathbf{v}}\cV(f)| \leq n-1$. Since furthermore $\dim~\Log|T_{\cV(f)}
(\mathbf{v})^\perp| = 1$ and $\Log|\cdot|$ is surjective it follows that $\dim~\Log|T_{\mathbf{v}}\cV(f)| = n-1$. Hence, for $\Ker(\Log|T_{\mathbf{v}}\cV(f)|) \subset T_{\mathbf{v}}\cV(f)$ it holds that $\dim~\Ker(\Log|T_{\mathbf{v}}\cV(f)|)$ $ = n-1$. Thus, we can choose an
orthogonal basis $B =(b_1,\ldots,b_{2n-2}) \subset \R^{2n}$ of $T_{\mathbf{v}}\cV(f)$ with $b_1,\ldots,b_{n-1} \in \Ker(\Log|T_{\mathbf{v}}\cV(f)|)$. 

Due to \eqref{Equ:Orthogonal} we have $\Ker(\Log|T_{\mathbf{v}}\cV(f)|) \subseteq \Arg(T_{\mathbf{v}}\cV(f))$, i.e., in particular,
$\Arg|_{\langle b_1,\ldots,b_{n-1}\rangle}$ is an immersion. Moreover, as $\dim~\Log|T_{\mathbf{v}}\cV(f)| = n-1$, also
$\Log|_{\langle b_n,\ldots,b_{2n-2}\rangle}$ is an immersion. Thus, with \eqref{Equ:Orthogonal} $b_n,\ldots,b_{2n-2}$ is contained in $\Ker(\Arg(T_{\mathbf{v}}\cV(f)))$, i.e., $\dim~\Ker(\Arg(T_{\mathbf{v}}\cV(f))) = n-1$ and therefore $\dim~\Arg(T_{\mathbf{v}}\cV(f)) = n-1$. Hence, $\mathbf{v}$ is critical under the $\Arg$ map. The converse of this argument follows analogously.
\end{proof}

\bibliographystyle{amsplain}
\bibliography{AmoebaBoundary}

\end{document}